\documentclass[11pt]{amsart}
\usepackage[margin=2.55 cm]{geometry}
 \usepackage[utf8]{inputenc}
 \usepackage{mathtools,amsfonts,amsthm,
 mathrsfs,amssymb,amsmath,bm}
\usepackage{textcomp}
\mathtoolsset{showonlyrefs}
\usepackage{bookmark}
\usepackage{enumerate}
\usepackage{microtype}
\usepackage [english]{babel}
\usepackage [autostyle, english = american]{csquotes}
\MakeOuterQuote{"}
\usepackage{todonotes}
\usepackage{ulem}

\definecolor{DarkDesaturatedBlue}{HTML}{3A3556}
\definecolor{VividOrange}{HTML}{F15918}
\definecolor{PureOrange}{HTML}{FFBA00}
\definecolor{LightGrayishPink}{HTML}{EEC5D5}
\definecolor{VerySoftBlue}{HTML}{B5AFDB}

\theoremstyle{plain}
\newtheorem{theorem}{Theorem}[section] 
\newtheorem{thm}[theorem]{Theorem}
\newtheorem{lemma}[theorem]{Lemma}

\newtheorem{proposition}[theorem]{Proposition}
\newtheorem{corollary}[theorem]{Corollary}
\newtheorem{conj}[theorem]{Conjecture}
\theoremstyle{definition}
\newtheorem*{definition*}{Definition}
\newtheorem{definition}[theorem]{Definition}

\newtheorem{hyp}{Hypothesis}

\newcommand{\cH}{\mathcal{H}}

\newcommand{\bX}{\mathbf{X}}
\newcommand{\bR}{\mathbf{R}}

\newcommand{\sC}{\mathscr{C}}

\newcommand{\bL}{\mathbf{L}}
\newcommand{\bG}{\mathbf{G}}
\newcommand{\bY}{\mathbf{Y}}
\newcommand{\bsigma}{{\boldsymbol{\sigma}}}
\newcommand{\btau}{{\boldsymbol{\tau}}}

\newcommand{\dist}{\text{dist}}
\newcommand{\ber}{\text{Ber}}

\newcommand{\eps}{\varepsilon}

\newcommand{\pr}[1]{\mathbb{P}\left[#1\right]}
\newcommand{\esp}[1]{\mathbb{E}\left[#1\right]}

\newcommand{\pth}[1]{\left(#1\right )}
\newcommand{\bigO}[1]{O\!\pth{#1}}

\newcommand{\ceil}[1]{\left\lceil #1 \right\rceil}

\newcommand{\restrict}[2]{{#1}_{|{#2}}}

\title{Uniformly Random Colourings of Sparse Graphs}

\author[E.\ Hurley]{Eoin Hurley}
\address{Unaffiliated}
\email{eoin.hurley@umail.ucc.ie}
\thanks{The research leading to these results was partially supported by the  Deutsche Forschungsgemeinschaft (DFG, German Research Foundation) -- 428212407 (E. Hurley)}
 \author[F.\ Pirot]{Fran\c{c}ois Pirot}
 \address{Université Paris-Saclay, France.}
\email{francois.pirot@lisn.fr}

\begin{document}

\pagenumbering{gobble}

\begin{abstract}
We analyse uniformly random proper $k$-colourings of sparse graphs with maximum degree $\Delta$ in the regime $\Delta < k\ln k $.
This regime corresponds to the lower side of the shattering threshold for random graph colouring, a paradigmatic example of the shattering threshold for random Constraint Satisfaction Problems. 
We prove a variety of results about the solution space geometry of colourings of fixed graphs, generalising work of Achlioptas, Coja-Oghlan \cite{achlioptas2008algorithmic}, and Molloy \cite{molloy2012freezing} on random graphs, and justifying the performance of stochastic local search algorithms in this regime.  
Our central proof relies only on elementary techniques, namely the first-moment method and a quantitative induction, yet it strengthens list-colouring results due to Vu \cite{Vu02}, and more recently Davies, Kang, P., and Sereni \cite{DKPS20+}, and generalises state-of-the-art bounds from Ramsey theory in the context of sparse graphs.
It further yields an approximately tight lower bound on the number of colourings, also known as the partition function of the Potts model, with implications for efficient approximate counting. 


\end{abstract}

\maketitle


\pagenumbering{arabic}


\section{Introduction}
Theoretical Computer Science (TCS), Statistical Physics, and Combinatorics are nearest neighbours in the network of the sciences.
Their stated aims typically differ, but occasionally they converge, leading to a fruitful exchange of ideas. 
A central example of this overlap is graph colouring; for TCS this is a constraint satisfaction problem (CSP), for Statistical Physics this is the anti-ferromagnetic Potts model, and for Combinatorics this is a $ 150$-year-old well of interesting problems.  
In particular, colourings of sparse graphs and random graphs of average degree $\Delta$ with $k\ln k = \Omega(\Delta)$ colours have proved to be a rich source of phase transition type behaviour.
Our focus is on uniformly random $k$-colourings of graphs of maximum degree $\Delta$, satisfying some sparsity conditions (e.g. triangle-freeness) and under the assumption that $\Delta< k\ln k$, in other words right below the phase transition. 

The bulk of the paper is centred around the TCS perspective. We build on the rich literature surrounding phase transitions in random CSPs and the so-called shattering threshold by proving that, above the shattering threshold, properties of the solution space geometry of random graphs --- namely those investigated in \cite{achlioptas2008algorithmic,A06,molloy2012freezing} --- in fact hold for any fixed sparse graph.
For Statistical Physics, we prove an approximately tight lower bound on the number of colourings of sparse graphs; this corresponds to a lower bound on the partition function of the $k$-state anti-ferromagnetic Potts model at temperature $0$. 
These bounds in turn suggest interesting possibilities for limits of sequences of graphs with growing girth, as well as efficient approximate counting  of colourings. 
For Combinatorics, our result also extends to local list-colourings, where list sizes are defined according to local parameters (namely the degree of each vertex), and allow for a certain density of edges within the neighbourhoods.
Our results in this direction strengthen all previous results, and do so with a simpler proof.
The bounds that can be derived from our work are tight barring a breakthrough on a $50$-year-old Ramsey theory problem. 

We now state a simplified version of our primary result (Theorem~\ref{thm:main}), from which all other results follow.
We write $[k]\coloneqq \{1,\dots,k\}$.
For every graph $G$ we let $\sC_k(G)$ be the set of proper $k$-colourings of $G$ (we omit the subscript when $k$ is clear from the context). 
For every vertex $v\in V(G)$, every subgraph $H\subseteq G$, and every $k$-colouring $\sigma \in \sC_k(H)$, we let $L_\sigma(v) \coloneqq [k] \setminus \sigma(N(v) \cap V(H))$ be the list of colours available for $v$ given $\sigma$, and we let $\ell_\sigma(v) \coloneqq |L_\sigma(v)|$ be its order.

\begin{thm}\label{thm:main_simp}
Let $\eps \in (0,1)$ be fixed, let $\Delta$ be sufficiently large (in terms of $\eps)$, and let $G$ be a triangle-free graph of maximum degree at most $\Delta$.
Fix $k$ such that $(1-\eps) k \ln k \ge \Delta$ and set $\ell \coloneqq \Delta^{\eps/2}$.
Then, for every $v\in V(G)$, the uniformly random proper $k$-colouring $\bsigma$ of $G\setminus v$ satisfies
$\esp{\ell_{\bsigma}(v) } \ge  \ell$.
\end{thm}

We observe that the conclusion of Theorem~\ref{thm:main_simp} also holds for the uniformly random proper $k$-colouring $\btau$ of $G$. Indeed, the number of extensions of $\sigma \in \sC(G\setminus v)$ to $\sC(G)$ is precisely $\ell_\sigma(v)$, and thus $\esp{\ell_\btau(v)} = \esp{\ell_\bsigma(v)^2} \ge \esp{\ell_\bsigma(v)}$.

The proof of Theorem~\ref{thm:main_simp} relies on the following Coupon-Collector Lemma which is a randomised version of a result due to Molloy \cite{Mol19}.

\begin{lemma}[Coupon-Collector Lemma]
\label{lem:coupon-collector}
Suppose we have random non-empty lists $\bL_1,\dots,\bL_d$, each of which takes values in the finite subsets of $\mathbb{N}$. 
Fix some integer $t\ge 1$, and define the random variable $\bX \coloneqq \#\{i \in [d] : |\bL_i| \le t\}$.
Now choose an element $\bsigma(i)$ of $\bL_i$ uniformly at random for each $i \in [d]$ and define the random variable $\bL\coloneqq [k] \setminus \{ \bsigma(i) : i \in [d]\}$. 
Then
\[
\esp{|\bL|} \ge k_0\,\mathrm{e}^{-\pth{1+\frac{1}{t}}\frac{d}{k_0}}, \quad \mbox{where } k_0 = k - \esp{\bX}.
\]
\end{lemma}

For exposition's sake, let us show how Theorem~\ref{thm:main_simp} follows from the Coupon-Collector Lemma.
The proof of the more general statement of Theorem~\ref{thm:main} uses the same idea, but in order to get the explicit value of $\eps$ and handle the more general set-up we need an extra layer of technicality that obscures its conceptual simplicity.
Let us mention that the proof relies on an induction hypothesis that is (a rephrasing of) the standard  Rosenfeld counting setup \cite{Ros20,WaWo20+}.
\begin{proof}[Proof of Theorem~\ref{thm:main_simp}]
    We prove the result by induction on $n$, the order of $G$.
    If $G$ has only $1$ vertex we are trivially done. 
    Suppose the statement holds for all $H$ satisfying the conditions of the statement and with order at most $n-1$. 
    Let $G$ be as in the statement with order $n$ and let $v\in V(G)$ be arbitrary.
    Letting $H\coloneqq G\setminus v$, we will show that the uniformly random $k$-colouring $\bsigma \in \sC(H)$ satisfies
    $\esp{ \ell_{\bsigma}(v) } \ge \ell$. Observe that this is equivalent to the statement $|\sC(G)| \ge \ell \,|\sC(G\setminus v)|$.

    Let us fix some integer $t$. Given the realisation of $\bsigma$, we say that a neighbour $u\in N(v)$ of $v$ has a \emph{short list} if $\ell_{\bsigma}(u) \le t$, and we let $S_\bsigma$ be the set of vertices $u\in N(v)$ with short lists.
    First we observe that the expected size of $S_\bsigma$ is small.
    %
    %
    Indeed, consider the probability that a vertex $u\in N(v)$ belongs to $S_\bsigma$; we have

    \begin{equation}
    \label{eq:list_tail_bound}
        \pr{\ell_\bsigma(u) \le t} = \frac{\#\{\sigma \in \sC(H) : \ell_\sigma(u)\le t\}}{|\sC(H)|} \le \frac{t\, |\sC(H\setminus u)|}{\ell\, |\sC(H\setminus u)|} = \frac{t}{\ell},
    \end{equation}
    where the bottom line of the inequality follows from the induction hypothesis applied on $H$ and $u$. 
    So $\esp{|S_\bsigma|} \le \Delta t/\ell$.

    Let us write $N(v) = \{u_1, \ldots, u_d\}$.
    We are now ready to apply Lemma~\ref{lem:coupon-collector} to the random lists $L_\bsigma(u_1), \ldots, L_\bsigma(u_d)$, with $\bX = |S_\bsigma|$ counting the number of short lists. 
    We simultaneously resample $\bsigma(u_i)$ uniformly at random from $L_\bsigma(u_i)$ for every $u_i \in N(v)$ 
    and note that this is precisely the set-up of   Lemma~\ref{lem:coupon-collector}.
    We obtain that $\esp{L_\bsigma(v)} \ge k_0 \mathrm{e}^{-\pth{1+1/t}d/k_0}$ where $k_0 = k-\esp{|S_\bsigma|} \ge k-\Delta t/\ell$.
    Crucially, the realisation of the lists $L_\bsigma(u_1), \ldots, L_\bsigma(u_d)$ is determined by the restriction of $\bsigma$ to $H\setminus N(v)$, so it is not affected by the resampling.
    It follows that $\bsigma$ remains uniformly distributed in $\sC(H)$ after this resampling. 
    To finish the proof, we may set $t\coloneqq \ceil{5/\eps}$, and assume that $\Delta$ is large enough so that $\Delta t/\ell \le \eps k/5$. We have

    \begin{align*}
        \esp{|L_\bsigma(v)|} &\ge \pth{k-\frac{\Delta t}{\ell}} \mathrm{e}^{-\frac{(1+1/t)\Delta}{k-\Delta t/\ell}} \ge \pth{1-\eps/5} k \mathrm{e}^{-\frac{(1+\eps/5)(1-\eps)k\ln k}{(1-\eps/5)k}} \ge (1-\eps/5)k\mathrm{e}^{(-1-\eps/5 + 4\eps/5)\ln k}\\
        &\ge (1-\eps/5)k^{3\eps/5}\ge \Delta^{\eps/2},
    \end{align*}
    assuming once again that $\Delta$ is large enough.

    
\end{proof}

The statement of Theorem~\ref{thm:main_simp} concerns only the first moment of list orders. We will show that, from this statement, we can derive many properties that the random proper $k$-colouring $\bsigma$ satisfies with high probability. To this, we need to prove concentration of certain random variables; this is usually achieved by expressing said variable as a sum of (approximately) negatively correlated binary random variables.
 Here, we wish to prove concentration for the number of short lists that appear in the neighbourhood of a vertex in a uniformly random colouring.
 We are unfortunately in a context where negative correlation does not hold and thus we introduce another property, namely Bernoulli-domination, that suffices to have high concentration inequalities in that context. This is obtained, in Theorem~\ref{cor:ind_set_corr}, by combining Theorem~\ref{thm:main_simp} with repeated applications of \eqref{eq:list_tail_bound}.

 This work constitutes, to our knowledge, the first rigorous analysis of the properties of a typical proper colouring of an arbitrary deterministic graph in this regime, and we suspect our methods may be adapted to answer many more questions about such typical colourings and CSP solutions.




\subsection{Organisation of Paper}
In Section~\ref{sec:context and statement}, we give problem-specific context and state our headline results. 
In Subsection~\ref{sec:CSP}, 
we review the random CSP's literature before introducing our results on the solution space geometry for colourings of sparse graphs. 
In Subsection~\ref{sec:local_list} we discuss the Combinatorics literature and the extensions to list-colourings and graphs of bounded local density. 
In Subsection~\ref{ref:partition functions} we state our lower bound on the number of colourings and discuss the relevance to graph limits and efficient approximate counting. 
In Section~\ref{sec:technical_results} we introduce 
the probabilistic machinery (the Coupon-Collector-type results in Subsection~\ref{sec:coupon-collector} and the ones using Bernoulli-domination in Subsection~\ref{sec:bernoulli-domination}) that we rely on, and apply it to sketch the proofs of our results from Subsection~\ref{sec:CSP} in Subsection~\ref{sec:proofs}. 
All remaining proofs lie in the Appendix.

\section{Context and Statement of Results}\label{sec:context and statement}
All graphs considered will be simple loopless graphs, denoted by $G = (V(G),E(G))$ or $H$ and typically on $n$ vertices. 
For a vertex $v \in V(G)$, we denote by $N_G(v)$ the neighbourhood of $v$ in $G$, and by $N_G[v]\coloneqq \{v\}\cup N_G(v)$ its closed neighbourhood. We omit the subscript if it is clear from the context.
For a subset $U\subseteq V(G)$ let $G[U]\subseteq G$ denote the subgraph of $G$ induced by $U$. 
Throughout, when we say that a property holds with high probability (w.h.p.), we mean with probability tending to $1$ as $n$ tends to $\infty$.
As is standard, we denote by $o_y(x)$ a real-valued function $f(x,y)$ such that $f(x,y)/x\rightarrow 0$ as $y \rightarrow \infty$ for all $x>0$.
Further, we denote by $\Omega(x)$ a real-valued function $f(x)$ for which there exists $C>0$ such that $f(x) \ge Cx$ for all $x$ sufficiently large.

\subsection{Solution Space Geometry}\label{sec:CSP}
For over half a century we have known that many Constraint Satisfaction Problems (CSPs) are NP-hard (in the worst case).
It was hoped that average case complexity would paint a brighter picture, and indeed early results suggested this might be so \cite{turner1988almost,dyer1986fast,kuvcera1989graphs}.
These results yielded polynomial-time algorithms for many random CSPs, including for $k$-colouring almost all graphs of chromatic number at most $k$.
However, a more challenging\footnote{More challenging and perhaps more natural, as the assumed condition is checkable in polynomial time.} test case soon emerged; finding a polynomial-time algorithm that, with high probability, finds a solution to a random CSP with a \textit{fixed ratio of constraints to variables}. 
When the CSP is graph colouring this test case becomes: $k$-colour the \emph{Erd\H{o}s-R\'enyi random graph $G_{n,dn/2}$}. 
That is, the graph chosen uniformly at random from all graphs with $n$ vertices and $dn/2$ edges, so of average degree $d$.
In \cite{achlioptas1997analysis} it was shown that w.h.p. the following algorithm suffices for $d \le k \ln k$:
\begin{enumerate}
    \item choose a vertex $v$ uniformly at random from those with the least number of available colours;
    \item colour $v$ uniformly at random with one of its available colours. 
\end{enumerate}
Remarkably, while $G_{n,dn/2}$ is known to be w.h.p. $k$-colourable for $d \sim 2 k\ln k $, no algorithm has been proven to succeed for any $d > (1+o_k(1))k\ln k$.
This extremely naive algorithm is the state of the art. 
As observed by Achlioptas and Coja-Oghlan in \cite{achlioptas2008algorithmic}, this state of affairs appears to be universal in that it holds for "nearly every random CSP of interest" (when $d$ represents the ratio of constraints to variable).
That is, extremely naive algorithms (known for decades) work up to some threshold ($d_{Sh}$), which is well below the SAT-UNSAT threshold ($d_{SU}$), but no algorithms are known to work for $d>d_{Sh}$.
The value of $d_{Sh}$ is not arbitrary; it coincides with the conjectured shattering threshold \cite{krzakala2007gibbs} from the statistical physics literature (which is defined in terms of expected spatial correlation decay).
Drawing inspiration from these conjectures, and building on \cite{A06}, Achlioptas and Coja-Oghlan rigorously demonstrated that, at least in the case of random graph colouring, random $k$-SAT and random $k$-uniform hypergraph $2$-colouring, this was no coincidence. 
As stated in earlier work \cite{A06}  a key idea underlying their approach is the following. 

\begin{hyp}\label{meta-hyp}
The geometry of the solution space of a given CSP instance dictates the performance of algorithms on that instance\footnote{This has a very similar flavour to the Overlap Gap Property, see \cite{gamarnik2021overlap}.} 
\end{hyp}

They proved that w.h.p. two phase transitions, one local and one global, occur in the geometry of the solution space  as $d$ increases past this shattering threshold ($d_{Sh}~k\ln k$ for random graph colouring).
\textit{Globally}, the solution space starts as a large well-connected ball before shattering into exponentially many, exponentially small, well-separated clusters.
\textit{Locally}, typical solutions start as able to change the assignment of almost any variable by changing $o(n)$ other variables (almost all variables are "loose") and end up with this not being the case for almost all variables (almost all variables are "rigid") (see Section~\ref{sec:RLFT}). 
They hypothesised that these phase transitions are \textit{the} barriers to efficient algorithms.
That is, these properties of the geometry of the solution space explain both the existence of algorithms for $d< d_{Sh}$ and the difficulty in finding such algorithms for $d_{SU}> d > d_{Sh}$.

In \cite{molloy2012freezing}, the local phase transition --- also known as the freezing threshold --- was further analysed in the case of random graph colouring, showing that an even stronger statement, Theorem~\ref{thm:mol-frozen}, holds. 
The local freezing threshold is of particular interest because, in \cite{krzakala2007gibbs,zdeborova2008statistical,zdeborova2007phase}, it (as opposed to the global clustering threshold)  was hypothesised to be the primary barrier to efficient algorithms.  
Indeed there is evidence that supports this; while the thresholds are asymptotically equal (in $k$), they are not for small $k$, and in \cite{achlioptas2002almost} a random graph colouring algorithm is provided for $d=4$, for values of $k$ above the clustering threshold but below the freezing threshold (the latter is always higher).  

Thus while random CSPs have proved to be challenging problems, there was a consistent explanation of why algorithms succeed or not, based on the typical geometry of the solution space of a random problem instance. 
However, a crack appeared in this picture when, in \cite{Mol19}, Molloy used entropy compression (inspired by Moser-Tardos \cite{moser2010constructive}) to prove that the following clever and simple algorithm yields a $k$-colouring of \textit{any} triangle-free graph $G$ of maximum degree $d \le (1-o_k(1))k\ln k$ (right up to the shattering threshold!):
\begin{enumerate}
    \item start from an arbitrary partial proper $k$-colouring of $G$;
    \item while there exists a $\textbf{Bad}$ vertex $v$, resample a partial proper colouring of the neighbourhood of $v$ uniformly at random\footnote{In fact in this step Molloy gave priority to the \textbf{Bad} vertices that were most recently not \textbf{Bad}, but that is not relevant to our discussion.};
    \item extend the partial colouring to a complete proper colouring. 
\end{enumerate}
This is of course a stochastic local search algorithm \cite{iliopoulos2019stochastic}. Molloy defined \textbf{Bad} vertices to be those with too few colours available, or too many uncoloured neighbours that might create a conflict on a given colour.
Via entropy compression, he demonstrated that after resampling neighbourhoods polynomially many times (in expectation), no vertices would be \textbf{Bad}. 
Then, a direct application of the algorithmic Lov\'{a}sz Local Lemma \cite{moser2010constructive} suffices to complete the colouring in polynomial-time. 
In fact, Bernshteyn \cite{Ber19} showed that the initial part of the algorithm could also be viewed as an application of the Lov\'asz Local Lemma \cite{alon2016probabilistic}.

The reason this constitutes a crack in our picture is that our barriers to the existence of efficient algorithms above the shattering threshold were hypothesised based on the analysis of random graphs.
But Molloy's algorithm works for \textit{all triangle-free graphs}, not just random graphs. 
Thus, working off of Hypothesis~\ref{meta-hyp}, we would expect that the geometry of the space of colourings of all triangle-free graphs of max degree $d < k \ln k$ is similar to that of $G_{n,d/n}$ as described by Achlioptas, Coja-Oghlan \cite{achlioptas2008algorithmic}. 
If, for example,  there existed a $d$-regular triangle-free graph for which the typical $k$-colouring was rigid, or for which the solution space consisted of exponentially many well-separated components, then this would rule out these phenomena (the local freezing threshold and the global clustering threshold respectively) as barriers to efficient algorithms.
The central contribution of this work, see Theorems~\ref{thm:recol} and~\ref{thm:girth-thawed}, is to show that no such graph exists in the former case, supporting the possibility that the local freezing threshold is a barrier to efficient algorithms.  
Establishing positive results on the connectivity of the solution space is also particularly relevant to understanding the performance of stochastic local search algorithms, such as Molloy's \cite{Mol19}.

\subsubsection{Rigid, Loose, Frozen, Thawed}\label{sec:RLFT}
In this section, we formally define the local "freezing" phase transition discussed in the introduction. 
For a discussion and formal definition of the global clustering phase transition see \cite[Section~2.1]{achlioptas2008algorithmic}.
We say "local" because it concerns the perspective from a single colouring, rather than viewing the whole solution space at once. 
For two colourings $\sigma$ and $\tau$ of a graph $G$, let $\dist(\tau,\sigma)\coloneqq \#\{v \in V(G): \sigma(v) \neq \tau(v)\}$ be the \emph{Hamming distance}. 
\begin{definition}
For $t,k\in \mathbb{N}$ and any graph $G$, we define \emph{the distance-$t$ $k$-colouring graph $\cH (G,k,t)$} as follows. 
The vertices of $\cH$ are the proper $k$-colourings of $G$ and two distinct vertices $\tau,\sigma$ are adjacent in $\cH$ if  $\dist(\tau,\sigma) \le t$.
Further, we define the \emph{$t$-clusters} to be the connected components of $\cH(G,k,t)$ (when $t=1$ we just say clusters). 
\end{definition}
Thus $\tau$ and $\sigma$ are in the same $t$-cluster if and only if there is a sequence of proper $k$-colourings $\tau=\tau_1,\dots,\tau_s = \sigma$ such that $\dist(\tau_i,\tau_{i+1})\le t$ for all $i \in [s-1]$ (note that $s$ may be arbitrarily large).
When $k$ and $t$ are clear from the context we will just write $\cH(G)$.

We now state the definitions of \cite{achlioptas2008algorithmic,molloy2012freezing} in these terms. 

\begin{definition}
Given $k,t \in \mathbb{N}$, a graph $G$, a vertex $v \in V(G)$ and $\tau\in \sC_k(G)$, we say that $v$ in $\tau$:
{\begin{itemize}
    \item is $t$-loose if for all $j\in [k]$ there exists a neighbour $\sigma\in \cH(G)$ of $\tau$ with $\sigma(v) =j$;
    \item is $t$-thawed if for all $j\in [k]$ there exists $\sigma \in \cH(G)$ in the same $t$-cluster as $\tau$ with $\sigma(v)=j$;
    \item is $t$-rigid if $\tau(v)=\sigma(v)$ for all neighbours $\sigma\in \cH(G)$ of $\tau$;
    \item is $t$-frozen if $\tau(v) =\sigma(v)$ for all $\sigma \in \cH(G)$ in the same $t$-cluster as $\tau$.
\end{itemize}}
\end{definition}


When we omit $t$ in the notations above, it means that $t=1$ by default. 
Note that the definitions are ordered from the most connected to the least connected, and that $t$-loose implies $t$-thawed and $t$-frozen implies $t$-rigid.  
Another way to view the definition of thawed and frozen is to consider the projections of the cluster containing $\tau$ onto $v$. 
If this projection contains one/all element(s) of $[k]$ then $v$ is frozen/thawed in $\tau$.
Note that loose/rigid are both about neighbours, not clusters; in particular a vertex $v$ can be loose only if it is isolated, and it is rigid if $N[v]$ spans all $k$ colours.

While in this paper we discuss uniformly random colourings $\bsigma$ of a fixed graph $G$, in \cite{achlioptas2008algorithmic,molloy2012freezing} they discussed uniformly random colourings of uniformly random graphs.
We make this precise as follows.
First fix $d,k\in \mathbb{N}$, and now consider $n\in \mathbb{N}$ and $m\coloneqq \frac{dn}{2}$ (we will let $n$ tend to infinity).
We now let $\bG$ be chosen uniformly at random from all graphs on $n$ vertices with $m$ edges.
Finally we let $\bsigma$ be a proper $k$-colouring of $\bG$ chosen uniformly at random. 
The random pair $(\bG,\bsigma)$ is the subject of \cite{achlioptas2008algorithmic,molloy2012freezing} (although they analyse it via the powerful "planted" model). 
Note that conditioned on $\bG=G$ our random variable $\bsigma$ is the same. 
Given $k$, $n$, and $d$, we will refer to the above described $(\bG,\bsigma)$ as the \emph{uniformly random instance-solution pair}. 

We now review what is known about rigid, loose and frozen variables (thawed was introduced in this paper). 
Let $k,d\in \mathbb{N}$ such that 
\begin{equation}\label{eq:shatter}
    (1+o_k(1))k\ln k \le d \le (2 - o_k(1))k\ln k.
\end{equation}
\begin{thm}\label{thm:ach-rigid}\cite[Theorem~4]{achlioptas2008algorithmic} For $k,d$ as in \eqref{eq:shatter}
 the uniformly random instance-solution pair $(\bG,\bsigma)$ w.h.p. contains at least $(1 - o_k(1)) n$ vertices that are $\Omega(n)$-rigid.  
\end{thm}
\begin{thm}\label{thm:mol-frozen}\cite[Theorem~2.4]{molloy2012freezing}
For $k,d$ as in \eqref{eq:shatter}, there exists $0<\alpha,\beta<1$ such that w.h.p. the uniformly random instance-solution pair $(\bG,\bsigma)$ contains $(\beta + o_n(1)) n $ vertices that are $\alpha n$-frozen. 
\end{thm}

These results make a strong statement about the solution space of a random graph for $k\ln k < d$.
In the same papers, the opposite side of the threshold was also described. 
Suppose
\begin{equation}\label{eq:unshattered}
    d \le (1-o_k(1))k \ln k.
\end{equation}
\begin{thm}\label{thm:ach-loose}\cite[Theorem~5]{achlioptas2008algorithmic}
For $k,d$ as in \eqref{eq:unshattered},
in the uniformly random instance-solution pair $(\bG,\bsigma)$ w.h.p. every vertex is $o(n)$-loose.
\end{thm}
\begin{thm}\label{thm:mol-loose}\cite[Theorem~2.4]{molloy2012freezing}
For $k,d$ as in \eqref{eq:unshattered}, there exists a constant $C>0$ such that w.h.p. the uniformly random instance-solution pair $(\bG,\bsigma)$ contains at most $o(n)$ vertices that are $(C \ln n)$-frozen. 
\end{thm}
We remark that Theorems~\ref{thm:ach-rigid} and~\ref{thm:ach-loose} also apply to $2$-colouring $k$-uniform hypergraphs and $k$-CNF. 
We build on Theorems~\ref{thm:ach-loose} and~\ref{thm:mol-loose}, showing that in fact that we can convert these into statements about uniformly random colourings of deterministic graphs of large enough girth. 

\subsubsection{Original Results on Solution Space Geometry}
Our first theorem is about triangle-free graphs, but note that it holds with probability $1$ as $d$ tends to infinity, rather than as $n$ tends to infinity. 
This is the best one could expect by only analysing the graph within a constant radius $r$ of a vertex $v$, because random fluctuations will occur with probability depending only on $r$ and $\Delta$. 
Analysing the graph at any radius larger than the girth is of course much more challenging.

\begin{thm}\label{thm:recol}
For all $\eps>0$ 
the following holds for all $k$ sufficiently large and $\Delta \le (1-\eps)k\ln k$. Let $G$ be a triangle-free graph of maximum degree at most $\Delta$ and let $\bsigma$ be the uniformly random proper $k$-colouring of $G$. 
Then for every vertex $v \in V(G)$ the following holds with probability at least $1-\mathrm{e}^{-\Delta^{\eps/3}}$:
\begin{enumerate}
    \item in $\bsigma$, $v$ is $(\Delta+1)$-loose;
    \item in $\bsigma$, $v$ is thawed.
\end{enumerate}
\end{thm}
In fact what we will show is that, with the above probability, one can change the colour of $v$ to any colour just by changing the colours of $v$ and its neighbours.

As mentioned above, due to the dependency of the probability on $d$, this does not directly strengthen Theorems~\ref{thm:ach-loose} and~\ref{thm:mol-loose}, although it does preclude the possibility of triangle-free graphs of maximum degree $\Delta$ with the properties guaranteed by Theorems~\ref{thm:ach-rigid} and~\ref{thm:mol-frozen}.
If we strengthen our girth condition, we do obtain strengthenings of  Theorems~\ref{thm:ach-loose} and~\ref{thm:mol-loose} (these are not possible for triangle-free graphs, see Section~\ref{sec:constructions}).

\begin{thm}\label{thm:girth-thawed}
 For all $\eps>0$ the following holds for all $k$ sufficiently large and $\Delta \le (1-\eps)k\ln k$. 
 Suppose $G$ is a graph on $n$ vertices with maximum degree $\Delta$ and girth at least $\ln\ln n$.
 Then for the uniformly random proper $k$-colouring $\bsigma$ of $G$, the following holds w.h.p. for every $v \in V(G)$:
 \begin{enumerate}
    \item in $\bsigma$, $v$ is $\bigO{(\ln n)^2}$-loose\footnote{With a more careful (and technical) analysis, the authors believe that this could be replaced with $O(\ln n)$-loose.};
    \item in $\bsigma$, $v$ is thawed.
\end{enumerate}
\end{thm}

Once again what we actually show is that we can change the colour of $v$ to whatever we please by only changing the colours of vertices at distance at most $O(\ln \ln n/\ln \Delta)$ from $v$.
Theorem~\ref{thm:girth-thawed} implies that the reconfiguration graph of the proper $k$-colourings of $G$ (where the reconfiguration step is to change the colour of any vertex) contains a connected component that covers all but $o_\Delta(1)$ colourings.
To see that these are indeed strenghtenings, recall that  with high probability (assuming that $n$ is large enough) there are no triangles and all cycles of length at most $o(\log_d n)$ are vertex disjoint in the Erd\H{o}s-R\'enyi random graph $G_{n,dn/2}$  (this is shown for the random regular graph in \cite[Equation~(2.8)]{mckay2004short}).
Thus, in the proof of Theorem~\ref{thm:girth-thawed}, when we consider the set of vertices at distance at most $o(\log_d n)$ from some vertex $v$, we can make the induced graph a tree by deleting a matching. 
This will not affect the analysis. 
Theorem~\ref{thm:girth-thawed} directly strengthens Theorem~\ref{thm:mol-loose}, by replacing $(C \ln n)$-frozen by $1$-frozen (in \cite{achlioptas2008algorithmic} it had already been remarked that $O(\ln n)$-looseness is possible).

We conjecture that the global connectivity properties proved in \cite{achlioptas2008algorithmic} should also hold for all triangle-free graphs. 
\begin{conj}
For $\eps,\Delta,k$ and $G$ as in Theorem~\ref{thm:main_simp}, there exists a cluster of $\cH(G)$ containing at least a $(1-o_\Delta(1))$-proportion of $\sC_k(G)$.
\end{conj}

\subsection{Local List Colouring and Density}\label{sec:local_list}
Given a  graph $G$, 
a \emph{list-assignment} of $G$ is a map $L\colon V(G) \rightarrow 2^{\mathbb{N}}$, and a \emph{proper $L$-colouring} $\sigma$ of $G$ is a map $c\colon V(G)\rightarrow \mathbb{N}$ such that $c(v)\in L(v)$ for every vertex $v \in V(G)$, and $c(u)\neq c(v)$ for every edge $uv\in E(G)$. Thus list-colouring is a CSP which includes colouring as a special case.
If there exists an $L$-colouring of $G$, we say that $G$ is \emph{$L$-colourable}.
The \emph{list chromatic number} of $G$, denoted $\chi_\ell(G)$, is the minimum $k$ such that $G$ is $L$-colourable for every list-assignment $L$ with $|L(v)|\ge k$ for every vertex $v\in V(G)$.
Note that the classical chromatic number $\chi(G)$ is always at most $\chi_\ell(G)$.

In a seminal result, Johansson \cite{Joh96} showed that the list chromatic number of triangle-free graphs of maximum degree $\Delta$ never exceeds $O(\Delta/\ln \Delta)$ as $\Delta \to \infty$. Two decades later, Molloy \cite{Mol19} showed with the help of entropy compression that this bound can be tightened to $(1+o(1))\Delta/\ln \Delta$ as $\Delta\to\infty$;  Bernshteyn \cite{Ber19} then showed that it was possible to replace the use of entropy compression with an application of the lopsided Lov\' asz Local Lemma and obtain a similar result. Following that breakthrough, there has been a lot of interest for triangle-free graph colourings and extensions \cite{AIS19, ABD21+, BKNP18+, DJKP20, DKPS20+}.

Although quite different, all the proofs in those works rely on a $2$-step colouring procedure. 
This involves first finding a partial proper colouring with appropriate properties and then extending it in a second step. 
This second step is usually called the Finishing Blow. 
In contrast, our proof does not require the use of the Finishing Blow, and is the first of that nature to the best of our knowledge. 
Our result also includes graphs where the density is bounded within each neighbourhood, a natural extension of triangle-free graphs. The first result in that vein is due to Alon, Krivelevich, and Sudakov \cite{AKS99}. 
They relied on the result of Johansson in order to prove that if a graph $G$ of maximum degree $\Delta$ is such that each neighbourhood spans at most $\Delta^2/f$ edges, where $1\le f \le  \Delta^2+1$, then its chromatic number is at most $O(\Delta/\ln f)$ as $f\to \infty$. This result was then extended to the list chromatic number by Vu \cite{Vu02}.
More recently, it has been proved by the second author together with Davies, Kang, and Sereni \cite[Corrolary 24]{DKPS20+}, that this upper bound can be tightened to $(2+\eps) \Delta/ \ln f$, provided that $f \ge (\ln \Delta)^{2/\eps}$, for every $\eps>0$.
Our main theorem is a strengthening of these results which is close to being optimal.

We were inspired by a counting argument which was first used in the context of graph colouring by Rosenfeld \cite{Ros20} (see \cite[Section 3.5]{Wo20} for an introduction to the method), and later devised more generally for hypergraph colouring by Wanless and Wood \cite{WaWo20+}.
We note however that our proof deviates from the standard arguments in order to exploit the relationship between uniformly random colourings of a graph $G$ and uniformly random colourings of induced subgraphs $H\subset G$.

Given a list-assignment $L\colon V(G)\to 2^\mathbb{N}$ of $G$, a proper $L$-colouring $\sigma$ of a subgraph $H$ of $G$ and a vertex $v\in V(G)$, we denote $L_\sigma(v) \coloneqq L(v) \setminus \sigma(N_H(v))$ the set of colours available at $v$ given $\sigma$, and $\ell_\sigma(v) \coloneqq |L_\sigma(v)|$ its size.

\begin{thm}\label{thm:main}
Let $G$ be a graph of maximum degree $\Delta$ such that every graph induced by a neighbourhood in $G$ has average degree at most $d\le \frac{\Delta}{6}-1$. Write $\rho \coloneqq \Delta/(d+1)$ and let $\ell \ge (d+1)(\ln \rho)^3$. Then for every list-assignment  $L\colon V(G)\to 2^\mathbb{N}$ with 
\[|L(v)| \ge \pth{1+\frac{2}{\ln \rho}} \frac{\deg(v)}{W\pth{\frac{\deg(v)}{\ell}}}\]
for every vertex $v\in V(G)$, the following holds.
For all $v \in V(G)$, the uniformly random proper $L$-colouring $\bsigma$ of $G\setminus v$ satisfies $\esp{\ell_\bsigma(v)} \ge  \ell$.
\end{thm}

As before, the conclusion also holds for the uniformly random $k$-colouring of $G$.
In the statement of Theorem~\ref{thm:main}, we use the $W$-Lambert function $z\mapsto W(z)$ which is defined as the reciprocal of the function $z \mapsto z\mathrm{e}^z$. In the proof of Theorem~\ref{thm:main}, we will use the well-known fact that $\mathrm{e}^{W(z)} = z/W(z)$. Moreover, we note that $W(z) = \ln z - \ln \ln z + o(1)$ as $z\to \infty$; hence, by fixing $d\coloneqq 0$, Theorem~\ref{thm:main} implies Theorem~\ref{thm:main_simp}. Since we have already given a proof of Theorem~\ref{thm:main_simp}, whose statement suffices for most applications in this paper, we defer the proof of Theorem~\ref{thm:main} to the Appendix.

In the context of Vu's result \cite{Vu02}, Theorem~\ref{thm:main} implies the following; this is proved in the Appendix.
\begin{corollary}
\label{cor:vu}
Let $G$ be a graph of maximum degree $\Delta$, such that every neighbourhood spans at most $\Delta^2/f$ edges, for some $1 \le f \le \Delta^2+1$. Then
$ \chi_\ell(G) \le (1+o(1)) \frac{\Delta}{\ln \min \{\Delta,f\}}$ as $f\to \infty$.
\end{corollary}
It has been observed in \cite{DJKP20} that there exist graphs satisfying the hypothesis of Corollary~\ref{cor:vu} and with chromatic number $(\frac{1}{2}-o(1))\Delta / \ln \min\{f,\Delta\}$. So the bound in Corollary~\ref{cor:vu} is sharp up to an asymptotic factor $2$. Reducing this gap would constitute a breakthrough in Ramsey theory, as this would imply an improvement of the estimate of the off-diagonal Ramsey numbers $R(3,t)$ for large $t\in \mathbb{N}$, a long-standing open problem.

We note that the results of this section also hold in the context of DP-colouring, an interesting extension of list colouring introduced by Dvo\v r\' ak and Postle \cite{DvPo18}, which is once again a CSP. 
We have decided to state those results in the context of list colouring in order to avoid the verbose formalism of DP-colouring, but the proof of the results in this section readily adapt to that context.

\subsection{Partition Functions and Approximation Schemes}\label{ref:partition functions}


Partition functions are important objects in statistical physics.
They count the weighted number of states of a system where said weight is proportional to the probability of seeing the system in said state. 
Thus they also serve as the normalising constant when we write the probability of seeing a particular state.
It seems natural to expect that our lower bound on the expected number of colours available at each vertex would yield a lower bound on the number of colourings and thus said partition function. 
By slightly adjusting the proof of Theorem~\ref{thm:main} we can show the following.

\begin{thm}\label{thm:count}
Let $G$ be an $n$-vertex graph of maximum degree $\Delta$ such that every graph induced by a neighbourhood in $G$ has average degree at most $d \le \frac{\Delta}{6}-1$. Let $f \coloneqq \Delta/(d+1)$ and
suppose $L\colon V(G)\to 2^\mathbb{N}$ is a list-assignment  with $|L(v)| \ge \pth{1+\frac{1}{\ln \rho}} q(v)$, where  
\[q(v) \ge \pth{1+\frac{1}{\ln \rho}} \frac{\deg(v)}{W\pth{\frac{\deg(v)}{(d+1)(\ln \rho)^3}}}\]
 for every vertex $v\in V(G)$.
Then there are at least $(q\big/\sqrt{D/(d+1)})^n$
 proper $L$-colourings of $G$, where $D$ is the geometric mean of the degrees in $G$, and $q$ is the geometric mean of $\{q(v)\}_{v\in V(G)}$.
\end{thm}

Shortly before an early version of this manuscript was published in a preprint repository,  independent work by Bernshteyn, Brazelton, Cao, and Kang appeared online \cite{BBRK21+}. 
They also rely on the counting argument used by Rosenfeld in \cite{Ros20}, 
and derive a lower-bound on $|\sC_k(G)|$ for every triangle-free graph $G$ of maximum degree $\Delta \le (1-\eps)k\ln k$ large enough in terms of $\eps$, and show that it is asymptotically sharp for $\Delta$-regular triangle-free graphs.
In particular, they prove that a random $\Delta$-regular triangle-free graph will almost surely admit no more than $(1-\frac{1}{k})^{\Delta n/2} \pth{\pth{1+\frac{2\ln n}{n}}k }^n$ $k$-colourings as $n \to \infty$.

\begin{thm}[Bernshteyn, Brazelton, Cao, Kang; 2021]
\label{thm:Anton}
For every $\eps>0$, there exists $\Delta_0$ such that the following holds. Let $G$ be an $n$-vertex triangle-free graph of maximum degree $\Delta\ge \Delta_0$ and with $m$ edges. Then, for every $k\ge (1+\eps) \Delta/\ln \Delta$, we have
$|\sC_k(G)| \ge (1-1/k)^m \big((1-\delta)k \big)^n$,
where $\delta=\frac{4}{k} \mathrm{e}^{\Delta/k}$.
\end{thm}

Let us pause to consider how many colourings we should expect. 
Suppose we choose the colour of each vertex independently  and uniformly at random. 
A moment's reflection shows that if there are no cycles in our graph then the events that two distinct edges are monochromatic are independent. 
Thus if the graph we are colouring has $n$ vertices, $m$ edges and no cycles, then the number of proper $k$-colourings is exactly $k^n(1-\frac1k)^m$.
Of course, no such graph exists in general, but we can still compute the number of colourings. 
If we consider the quantity $f(G,k) \coloneqq \frac 1n \ln |\sC_k(G)|$, known as the \emph{free energy per variable}, then we see that for a
 $\Delta$-regular "graph" $G$ with no cycles  we have $ \frac mn =\frac{\Delta}{2}$ and $f(G,k) = \ln(k (1-\frac1k)^{\Delta/2})$.
 Such a graph exists only if we allow an infinite number of vertices, and is unique; it is the infinite $\Delta$-regular tree $\mathbb{T}_\Delta$ (if you enjoy this perspective please see \cite{csikvari2016extremal}).
Of course, the number of colourings is also infinite, but we may extend the notion of free energy per variable to infinite graphs, and for $\mathbb{T}_\Delta$ this yields $f(\mathbb{T}_\Delta,k) \coloneqq \ln(k (1-\frac1k)^{\Delta/2}) \approx \ln k - \frac{\Delta}{2k}$.

Let us shift now to considering $ h(G,k) \coloneqq \frac{f(G,k)}{f(\mathbb{T}_\Delta,k)}$ for $\Delta$-regular graphs $G$. 
This measures the free energy per variable of $G$ relative to $\mathbb{T}_\Delta$. If $G$ has many colourings then $h(G,k)$ is larger than $1$; if it has few then $h(G,k)$ is smaller than $1$.  
Theorems~\ref{thm:count} and~\ref{thm:Anton} yield the bound\footnote{This is conjectured to hold with half the number of colours in \cite{Ber19}.} $ h(G,k) \ge (1-o_\Delta(1))$.
Further for the random regular graph  $G_{n,\Delta}$, \cite{Ber19} shows that almost surely
$h(G_{n,\Delta},k) \le (1+o_n(1))$.
We conjecture that both of these bounds are asymptotically tight for sequences of graphs with growing girth.

\begin{conj}
    For each $\Delta\in \mathbb{N}$ fix a  sequences $\{G^\Delta_i\}_{i=1}^\infty$ of $\Delta$-regular graphs with girth tending to $\infty$.
Then for all functions $k\colon \mathbb{N}\to \mathbb{N}$ satisfying $k(\Delta) > (1+\eps)\frac{\Delta}{\ln \Delta}$ for some $\eps>0$ it holds that
\[
\sup_{\Delta \in \mathbb{N}}\limsup_{i\rightarrow \infty} h(G^\Delta_i,k(\Delta)) \le \lim_{\Delta\rightarrow \infty}\liminf_{i\rightarrow \infty} h(G^\Delta_i,k(\Delta))=  1.
\]
\end{conj}

We have established the lower bound and we now chart a path toward the upper bound. 
First note that, writing $i(G)$ for the number of independent sets in $G$, it was shown in \cite{csikvari2016extremal} that if each graph in the aforementioned sequence  $\{G_i\}_{i=1}^\infty$ is bipartite, then $\lim_{j\rightarrow \infty} i(G_j)$ exists (it corresponds in a meaningful way to the infinite regular tree). 
We conjecture that the same result should hold when we replace $i(G)$ by $f(G,k)$ for $k \ln k \ge (1+o_\Delta(1))\Delta$ and that the limit should be $f(\mathbb{T}_\Delta,k)$ (again corresponding to the infinite regular tree)\footnote{Indeed it was already shown in \cite{csikvari2016sidorenko} that $f(\mathbb{T}_\Delta,k)$ is a lower bound for bipartite graphs, and thus our lower bounds can be viewed as an approximate cousin of \cite[Theorem~8]{sah2020reverse}.}.
Secondly, in \cite{sah2020reverse} it was shown that for all $k$ and $\Delta$-regular graphs $G$ on $n$ vertices we have
$f(G,k) \le f(K_{\Delta,\Delta},k)$, where $K_{\Delta,\Delta}$ is the complete $\Delta$-regular bipartite graph.
That is, it was shown that among $\Delta$-regular graphs, complete bipartite graphs maximise the free energy per variable (for $k$-colourings).
We conjecture that for all $g$, among $\Delta$-regular graphs $G$ with girth at least $g$, the \emph{supremum} of $f(G,k)$ is always achieved by a sequence of bipartite graphs\footnote{Of course the most elegant solution would be if $G_i$ was just a smallest (in terms of vertices) $\Delta$-regular bipartite graph with girth at least $g$ for all $i$, but this is stronger than what we need, nevermind that finding such graphs is an extremely challenging open problem \cite{furedi2006turan}}.
In fact one only needs to show that the supremum over bipartite graphs of girth $\omega(g)$ is at least that of general graphs of girth $g$ for some $\omega(g)$ tending to infinity. 
If both of these conjectures are true, then the limit for bipartite graphs would be an upper bound for all sequences of graphs with girth tending to $\infty$. 

We note that this would imply that  for all $\eps > 0$ there exists $\Delta_0$ such that for all $\Delta>\Delta_0$ we have an efficient approximation scheme for $f(G,k)$ for $\Delta$-regular graphs of large girth within a factor of $(1+\eps)$.
On the other hand in \cite{galanis2016inapproximability} it was shown that for fixed $\Delta$ there exists no Fully Polynomial-Time Randomised Approximation Scheme for approximating $f(G,k)$ within an additive error of $\eps$ unless RP = NP. 
It would be very interesting to pin down exactly where algorithmic complexity appears. 
We briefly remark that work for graphs of high girth and infinite trees was the launchpad for efficient approximate counting of graph $k$-colourings when $k$ is some constant factor larger than the maximum degree $\Delta$ \cite{bandyopadhyay2008counting,gamarnik2012correlation,weitz2006counting}.
Perhaps, under appropriate girth conditions, related approximation schemes are possible for $k< \Delta$. 

\section{Technical Results and Selected Proofs}\label{sec:technical_results}
In this section we collect our core results about uniformly random colourings and finish by proving Theorem~\ref{thm:recol} and sketching the proof of Theorem~\ref{thm:girth-thawed}.
\subsection{Coupon-Collector-type results}
\label{sec:coupon-collector}

We begin with a proof of the Coupon-Collector Lemma, of which we recall the statement hereafter.

\begin{lemma}[Coupon-Collector Lemma]
Suppose we have random non-empty lists $\bL_1,\dots,\bL_d$, each of which takes values in the finite subsets of $\mathbb{N}$. 
Fix some integer $t\ge 1$, and define the random variable $\bX \coloneqq \#\{i \in [\Delta] : |\bL_i| \le t\}$.
Now choose an element $\bsigma(i)$ of $\bL_i$ uniformly at random for each $i \in [d]$ and define the random variable $\bL\coloneqq [k] \setminus \{ \bsigma(i) : i \in [d]\}$. 
Then
\[
\esp{|\bL|} \ge k_0\,\mathrm{e}^{-\pth{1+\frac{1}{t}}\frac{d}{k_0}}, \quad \mbox{where } k_0 = k - \esp{\bX}.
\]
\end{lemma}

\begin{proof}
Let us fix a realisation $L_1, \ldots L_d$ of the random lists. 
We let $S = \{i : |L_i| \le t\}$ be the set of short lists, and $B = [d]\setminus S$ be the set of big lists.
We observe that since $|L_i| \ge t+1$ for every $i\in B$, we have $\frac{1}{|L_i|-1} \le (1+\frac{1}{t})\frac{1}{|L_i|}$. Combining this observation with one due to Molloy \cite{Mol19}, we obtain
    \begin{equation}
        \label{eq:doublesum}
        \sum_{x\in [k]} \sum_{\substack{i\in B \\ x\in L_i}} \frac{1}{|L_i|-1} \le \pth{1+\frac{1}{t}}\sum_{x\in [k]} \sum_{\substack{i\in B \\ x\in L_i}} \frac{1}{|L_i|} \le \pth{1+\frac{1}{t}} \sum_{i\in B} \sum_{x\in L_i} \frac{1}{|L_i|} \le 
        \pth{1+\frac{1}{t}}d.
    \end{equation}
The observation of Molloy is that each list $L_i$ appears in the first double sum exactly $|L_i|$ times, and it does so with weight $1/|L_i|$, meaning its contribution is always $1$. 

Let us fix $\bsigma(i) = x_i \in L_i$ for every $i\in S$, and let $L_0 \coloneqq [k] \setminus \{x_i : i\in S\}$. Note that $|L_0|\ge k-|S|$. 
We may now pick $\bsigma(i)$ uniformly at random from $L_i$ for every $i\in B$, and let $\bL_0 \coloneqq L_0 \setminus \{ \bsigma(i) : i\in B\}$. Note that $\bL_0$ is precisely $\bL$ under the condition that $\bsigma(S) = \{x_i : i\in S\}$. We have

    \begin{align*}
        \esp{|\bL_0|} &= \sum_{x\in L_0} \pr{x\notin \bsigma(B)} = \sum_{x\in L_0} \prod_{\substack{i\in B \\ x\in L_i}} \pth{1-\frac{1}{|L_i|}} \\
        &\ge \sum_{x\in L_0} \exp \pth{-\sum_{\substack{i\in B \\ x\in L_i}} \frac{1}{|L_i|-1}} &
        \mbox{since $1-\frac{1}{z} > \mathrm{e}^{-\frac{1}{z-1}}$ for every $z>1$;}\\
        &\ge |L_0| \exp\pth{-\frac{1}{|L_0|}\sum_{x\in L_0}\sum_{\substack{i\in B \\ x\in L_i}} \frac{1}{|L_i|-1}} &
        \mbox{by convexity of $\exp$;}\\
        &\ge (k-|S|) \exp\pth{- \frac{\pth{1+\frac{1}{t}}d}{k-|S|}} &
        \mbox{by $|L_0| > k -|S|$ and \eqref{eq:doublesum}.}
    \end{align*}

    It remains to average over all possible realisations of $\bL_1, \ldots, \bL_d$. We use Jensen's inequality together with the convexity of the function $z\mapsto z\mathrm{e}^{-C/z}$ for every $C>0$ over the interval $(0,+\infty)$, and obtain that 
    \begin{align*}
        \esp{|\bL|} \ge \esp{(k-\bX) \exp\pth{- \frac{\pth{1+\frac{1}{t}}d}{k-\bX}}} \ge \esp{(k-\bX)}\exp\pth{- \frac{\pth{1+\frac{1}{t}}d}{\esp{k-\bX}}} = k_0 \mathrm{e}^{-\pth{1+\frac{1}{t}}\frac{d}{k_0}},
    \end{align*}
    where $k_0 = k - \esp{\bX}$.

\end{proof}

A straightforward application of this result yields the following (deterministic) existence of a colouring that induces a big list at $v$, via the first moment method.
This can be obtained by following the exact same computation as that in the proof of Theorem~\ref{thm:main_simp}.

\begin{corollary}\label{cor:find_good_col}
Let $\eps \in (0,1)$ be fixed and let $\Delta$ be sufficiently large (in terms of $\eps$). Fix $k$ such that $(1-\eps)k\ln k \ge \Delta$. Let $L_1, \ldots ,L_\Delta \subset [k]$ be given, and suppose that the number of short lists is $\#\{i : |L_i| \le 5/\eps\} \le \eps k/5$.
Then there exists $\sigma(i) \in L_i$ for $i \in [\Delta]$ such that $|[k] \setminus \{\sigma(i) : i\in [\Delta]\}| \ge \Delta^{\eps/2}$.
\end{corollary}

\subsection{Bernoulli-domination}
\label{sec:bernoulli-domination}

For the rest of the proofs, we will need high-concentration bounds, and to that end we introduce the concept of \emph{Bernoulli-domination}. Let us first state the following useful Chernoff-Hoeffding type tail bound for binary random variables that are dominated by independent random variables. 
This is a special case of \cite[Theorem~3.4]{panconesi1997randomized}.
\begin{thm}
\label{thm:concentration}
Suppose $\bX_1,\dots,\bX_s$ are binary random variables and $\bY_1,\dots,\bY_s$ are independent binary random variables. Let $\bY \coloneqq \sum_{i\in [s]}\bY_i$ and $\bX\coloneqq \sum_{i\in [s]} \bX_i$. Write $ \mu \coloneqq\esp{\bY}$. Then if 
\[
\esp{\prod_{i\in J}\bX_i} \le \esp{\prod_{i\in J}\bY_i} ,
\]
for all $J \subset [s]$, it follows that for all $\delta>0$,
\[
\pr{\bX \ge (1+\delta)\mu} \le \left[\frac{\mathrm{e}^{\delta}}{(1+\delta)^{1+\delta}}\right]^\mu.
\]
In particular, for every $\sigma \ge 6\mu$, we have
\[ \pr{\bX \ge \sigma} \le \mathrm{e}^{-\sigma}.\]
\end{thm}

Motivated by the above theorem in the case where each $\bY_i$ follows the law of Bernoulli $\ber(p)$, we introduce the following definition.
\begin{definition}
We say binary random variables $(\bX_1,\dots,\bX_s)$ are \emph{$\ber(p)$-dominated} if 
\[
\esp{\prod_{i\in J}\bX_i} \le p^{|J|},
\]
for all $J \subseteq [s]$. 
\end{definition}

One can understand Theorem~$\ref{thm:concentration}$ as saying that the upper tail of the random variable is almost as small as that of a binomial random variable.
In fact, one can bootstrap Theorem~\ref{thm:main_simp} to make this analogy even stronger, by showing that large upward deviations on disjoint sets are also $\ber(p)$-dominated for some $p$.
The random variable $\bR_i$ in the next theorem is the indicator function of a large upward deviation on some set $Q_i\subset [s]$.

\begin{corollary}\label{cor:renorm_ber}
Suppose $(\bX_1,\dots,\bX_s)$ are $\ber(p)$-dominated. 
Let disjoint subsets $Q_1,\dots,Q_m \subseteq [s]$ of equal order be given, fix $\delta > 0$ and define the binary random variables 
\[
\bR_i \coloneqq \left\{\sum_{j\in Q_i}\bX_j > (1+\delta)p |Q_i| \right\},
\]
for $i \in [m]$. Let $q$ be the upper bound on $\pr{\bR_i = 1}$ given by Theorem~\ref{thm:concentration}. 
Then $(\bR_1,\dots,\bR_m)$ are $\ber(q)$-dominated.

\end{corollary}

\begin{proof}
Consider $J\subseteq [m]$; we wish to show that 
\[
\esp{\prod_{i\in J}\bR_i} \le q^{|J|}. 
\]
To this end, let $I\coloneqq \cup_{j \in J}Q_j \subset [s]$ and observe that, by assumption and Theorem~\ref{thm:concentration}, we have
\[
\pr{\sum_{i \in I}\bX_i > (1+\delta)p|I|} \le \left[\frac{\mathrm{e}^{\delta}}{(1+\delta)^{1+\delta}}\right]^{p|I|} = \prod_{j \in J} \left[\frac{\mathrm{e}^{\delta}}{(1+\delta)^{1+\delta}}\right]^{p|Q_j|} = q^{|J|}.
\] 
To conclude we observe that if $\bR_j =1 $ for all $j \in J$ then it follows that $\sum_{i \in I}\bX_i > (1+\delta)p|I|$.
\end{proof}

We now show that the appearance of short lists at vertices in an independent set is $\ber(p)$-dominated for an appropriate $p$. This is crucial for both Theorems~\ref{thm:recol} and~\ref{thm:girth-thawed}, and this could have many applications in other works.
Theorem~\ref{cor:ind_set_corr} can be seen as a strengthening of \eqref{eq:list_tail_bound} from Theorem~\ref{thm:main_simp}. In Theorem~\ref{thm:main_simp}, we could derive \eqref{eq:list_tail_bound} from one application of the induction hypothesis; going deeper into the induction allows us to derive the following stronger statement.

\begin{thm}\label{cor:ind_set_corr}
Let $\eps,\Delta,k,\ell$ and $G$ be as in Theorem~\ref{thm:main_simp}. Let $0<p<1$ and let $I$ be an independent set of $G$. Let $\bsigma$ be a uniformly random proper $k$-colouring of $G$.
Then the random variables $\bX_v\coloneqq \{ \ell_{\bsigma}(v) \le p\ell\}$ for $v \in I$ are $\ber(p)$-dominated.
\end{thm}
\begin{proof}
Let $J\subseteq I$, and label the vertices of $J$ by $\{v_i\}_{i\in [s]}$.
For every graph $H\subseteq G$, we may apply  Theorem~\ref{thm:main_simp} and obtain that, for all $v \in V(H)$, if $\bsigma$ is drawn uniformly at random from $\sC(H\setminus v)$, then $\esp{\ell_{\bsigma}(v)}\ge \ell$.  
In other words, we have $|\sC(H)|\ge \ell |\sC(H\setminus v)|$. 
Repeated applications of this identity yield that 
$|\sC(G)| \ge |\sC(G\setminus J)|\ell^s$. 
So we have 
\begin{align}
    \esp{\prod_{v\in J} \bX_v}
    =\frac{\#\{\sigma \in \sC(G): \ell_{\sigma}(v_1),\dots,\ell_{\sigma}(v_s)\le p\ell\}}{|\sC(G)|}  
    \le \frac{|\sC(G\setminus J)|\cdot (p\ell)^s}{|\sC(G\setminus J)|\cdot \ell^s} = p^s.
\end{align}
The upper bound for the numerator in the above inequality comes from two facts that rely on $J$ being an independent set. First, given a colouring $\sigma_0 \in \sC(G\setminus J)$, the number of extensions of $\sigma_0$ to a colouring $\sigma \in \sC(G)$ is $\prod_{v\in J} \ell_{\sigma_0}(v)$. Second, given such an extension $\sigma$, we have $\ell_\sigma(v)=\ell_{\sigma_0}(v)$ for every $v\in J$. 
We conclude that $(\bX_v)_{v\in I}$ are $\ber(p)$-dominated.
\end{proof}

Leveraging Theorem~\ref{cor:ind_set_corr}, and the fact that neighbourhoods induce independent sets, we can prove the following exponential upper bound on the likelihood of short lists. 
It also requires Lemma~\ref{lem:MOL-list-conc}, the proof of which is adapted from that of \cite[Lemma 7]{Mol19} and is deferred to the Appendix.

\begin{lemma}
\label{lem:MOL-list-conc}
Let $G$ be a triangle-free graph, let $v \in V(G)$ and let $\sigma_0$ be a proper $k$-colouring of $G\setminus N[v]$ (for some $k$), with at least one extension to $G$. 
Then if $\bsigma$ is the uniformly random extension of $\sigma_0$ to $G$, writing $\ell \coloneqq \esp{\ell_{\bsigma}(v)}$, we have 
\[
\pr{\ell_{\bsigma}(v) \le (1-\delta)\ell} \le \mathrm{e}^{-\frac{\delta^2\ell}{2}},
\]
for all $\delta \in (0,1)$.
\end{lemma}

\begin{corollary}\label{cor:exp-list-conc}
Let $\eps,\Delta,k,\ell$ and $G$ be as in Theorem~\ref{thm:main_simp}. Then for all $\delta \in (0,1)$ and $v \in V(G)$
\[
\pr{\ell_{\bsigma}(v) < (1-\delta)\ell } \le 2\mathrm{e}^{-\frac{\delta^2\ell}{2}},
\]
where $\bsigma$ is the uniformly random $k$-colouring of $G\setminus v$.
\end{corollary}

\begin{proof}
For every $u\in N(v)$, we let $B_u$ be the random event that $u$ has a short list, i.e. $\ell_\bsigma(u) \le t$ for $t\coloneqq \lceil 5/\eps \rceil$.
Since $G$ is triangle-free, $N(v)$ is an independent set. 
So by Theorem~\ref{cor:ind_set_corr}, the events $(B_u)_{u\in N(v)}$ are $\ber(t/\ell)$-dominated. We apply Theorem~\ref{thm:concentration}, and obtain that the probability that more than $6t\Delta/\ell$ neighbours of $v$ have a short list is at most $\mathrm{e}^{-6t\Delta/\ell}$.

Now, let $\sigma_0$ be a possible realisation of $\restrict{\bsigma}{G\setminus N[v]}$ such that no more than $6t\Delta/\ell$ neighbours of $v$ have a short list in $\sigma_0$. When $\Delta$ is large enough, this is at most $\eps k/5$. In that case, we can repeat the computation in the proof of Theorem~\ref{thm:main_simp}, and obtain that 
\[ \esp{\ell_\bsigma(v)\mid \restrict{\bsigma}{G\setminus N[v]} = \sigma_0} \ge \ell.\]
By applying Lemma~\ref{lem:MOL-list-conc},
we obtain that 
\[ \pr{\ell_\bsigma(v) \le (1-\delta)\ell \mid\restrict{\bsigma}{G\setminus N[v]}=\sigma_0} \le \mathrm{e}^{-\frac{\delta^2\ell}{2}}.
\]
Overall, the probability that $\ell_\bsigma(v) \le (1-\delta)\ell $ is therefore at most $\mathrm{e}^{-6t\Delta/\ell}+\mathrm{e}^{-\frac{\delta^2\ell}{2}} \le 2\mathrm{e}^{-\frac{\delta^2\ell}{2}}$
\end{proof}

It is interesting to observe that it is not possible to have a stronger form of Theorem~\ref{cor:ind_set_corr} where we replace Bernoulli-domination with negative correlation. Indeed, given a graph $G$, if $u,v\in N(v)$ share the same neighbourhood, then the lists of $u$ and $v$ are perfectly correlated.




\bigskip

We finish this subsection by proving the following bootstrap percolation result which may be of independent interest. 
The setup is as follows. 
Consider the rooted  $\Delta$-ary tree $T$ of depth $f$. 
That is, the tree constructed by starting from the root $r$, adding $\Delta$ \emph{children}, and then adding $\Delta$ children to each of the leaves, and repeating this a total $f$ times, so that the distance from $r$ to each of the $\Delta^f$ leaves is exactly $f$.
At step $1$, we randomly activate a subset of the leaves. At each step $i\ge 2$, we activate a vertex if at least $s$ of its children have been activated at step $i-1$.
We call this process \emph{$s$-upward percolation}.
Clearly, this process reaches a stable state in at most $f$ steps. 
We will be interested in the probability that the root $r$ is activated at the end of the process. 
Note that once we have chosen which leaves to activate, the process is deterministic. 
We are interested in the situation where the activation probabilities for the leaves are $\ber(p)$-dominated. 

We note that it is essential that the activation events are $\ber(p)$-dominated. 
Indeed, an adversary only needs to activate $s^f$ leaves in order to activate the root and if $s<\Delta$, then this is a vanishing proportion of all leaves as $f\rightarrow \infty$.
However,  in order to do so the adversary's activated leaves must be the leaves of  an $s$-ary tree. 
This a very low entropy strategy,
and we can use the renormalisation properties of $\ber(p)$-dominated random variables  (Corollary~\ref{cor:renorm_ber}) to show that this is very unlikely.

\begin{lemma}\label{lem:percolation}
Let $0<p<1$ be a real value, and 
let $\Delta \ge 2$, $s \ge \max \{6p\Delta, 3\ln \Delta\}$, and $f\ge 1$ be integers.
Suppose we perform $s$-upward percolation on a $\Delta$-ary tree $T$ of depth $f$ rooted in a vertex $r$, where the events that the leaves are activated are $\ber(p)$-dominated. 
Then we have\footnote{With a more careful analysis one can replace $\ceil{f/2}$ by $(1-o_f(1))f$.}
\[
\pr{ \mbox{$r$ is activated}} \le \exp\pth{-s^{\ceil{f/2}}}. 
\]
\end{lemma}

\begin{proof}

Given a vertex $v\in V(T)$, let $T[v]$ denote the unique $\Delta$-ary subtree of $T$ rooted at $v$. 
For every $i\le f$, we let $V_i(T)$ denote the set of nodes at depth $i$ in $T$ (in particular, we have $V_0(T) = \{r\}$).
If $r$ is activated, then there is an $s$-ary subtree of $T$ of depth $f$ that contains only activated vertices. In particular, at each depth $i\le f$, there are at least $s^i$ activated vertices. We conclude that, for every fixed $i\le f$, we have
\[ \pr{\mbox{$r$ is activated}} \le \pr{\mbox{$V_i(T)$ contains at least $s^i$ activated vertices}}.\]

Let $T_1, \ldots, T_n$ be a collection of disjoint trees of depth $f$, of respective roots $r_1, \ldots, r_n$. 
We perform $s$-upward percolation on each of them, and assume that the events that the leaves are activated are $\ber(p)$-dominated. Denote by $A_j$ the random event that $r_j$ is activated, for every $j\in [n]$. 
We will show by induction on $f$ that the events $(A_j)$ are $\ber(q)$-dominated, for $q=\exp(-s^{\ceil{f/2}})$.
For the base case $f=1$, for every $j\in [n]$ we have
\[ \pr{\mbox{$r_j$ is activated}} \le \pr{\mbox{$V_1(T_j)$ contains at least $s$ activated vertices}} \le \mathrm{e}^{-s} = q,\]
by Theorem~\ref{thm:concentration} applied with $\mu \coloneqq p\Delta$ and $\sigma \coloneqq s\ge 6\mu$.
We are in the setting of Corollary~\ref{cor:renorm_ber}, so the events $(A_j)$ are $\ber(q)$-dominated.

We now assume that $f\ge 2$, and we fix $i\coloneqq \ceil{f/2}$. Let us consider any tree $T_j$ for $j\in [n]$.
For every $x\in V_i(T_j)$, let $A_x$ be the event that $x$ is activated.
We apply induction on the collection of trees $T_j[x]$ of depth $f-i \ge (f-1)/2$, for $x\in V_i(T_j)$, and obtain that the events $(A_x)_{x\in V_i(T_j)}$ are $\ber(q')$-dominated, with $q' = \exp(-s^{\ceil{(f-1)/4}})$.
We apply Theorem~\ref{thm:concentration} with $\mu \coloneqq q'\Delta^i$ and $\sigma \coloneqq s^i$ and obtain that 
\[ \pr{\mbox{$r_j$ is activated}} \le \pr{\mbox{$V_i(T_j)$ contains at least $s^i$ activated vertices}} \le \mathrm{e}^{-s^i},\]
if we can prove that $\sigma\ge 6\mu$. We are again in the setting of Corollary~\ref{cor:renorm_ber}, so the events $(A_j)_{j\in [n]}$ are $\ber(q)$-dominated with $q=\exp(-s^{\ceil{f/2}})$, as desired.

We now prove that we have $\sigma\ge 6\mu$.
If $f \le 5$, we have $q' = \mathrm{e}^{-s}$ and $i\le 3$, so this reduces to $(s/\Delta)^3 \ge 6\mathrm{e}^{-s}$. Since $s^3 \ge (3\ln 2)^3 > 6$, it suffices to prove that $\mathrm{e}^s \ge \Delta^3$, which holds by assumption on $s$.
If $f>5$, it suffices to prove that $s^{\ceil{(f-1)/4}} \ge \ceil{f/2}\ln \Delta$. Since $s\ge 3\ln \Delta \ge 2$, it suffices to prove that $2^{\ceil{(f-5)/4}} \ge \frac{1}{3}\ceil{f/2}$. It is easy to check that this holds for every integer $f>5$.
\end{proof}

\subsection{Proofs of Theorem~\ref{thm:recol} and Theorem~\ref{thm:girth-thawed}}
\label{sec:proofs}

Relying on the tail-bound given by Corollary~\ref{cor:exp-list-conc} for the probability of having a short list, we can prove Theorem~\ref{thm:recol}.
\begin{proof}[Proof of Theorem~\ref{thm:recol}]
Fix a vertex $v\in V(G)$, and let $\bsigma$ be the uniformly random $k$-colouring of $G\setminus v$.
We will actually prove that, with high probability (with respect to $\Delta$), we have $\ell_\bsigma(u)\ge 2$ for every $u\in N(v)$. Under that condition, for every colour $x\in [k]$, we may sequentially resample $\bsigma(u)$ from $L_\bsigma(u)\setminus x$ for every $u\in N(v)$ (since $N(v)$ is an independent set, the lists $(L_\bsigma(u))_{u\in N(v)}$ are not affected by this resampling). Then we can set $\bsigma(v)\gets x$. This proves that $v$ is both $(\Delta+1)$-loose and thawed.

For every $u\in N(v)$, we let $B_u$ be the random event that $u$ has a short list, i.e. $\ell_\bsigma(u)\le \ell/2$, where $\ell =\Delta^{\eps/2}$.
By Corollary~\ref{cor:exp-list-conc} we have 
$\pr{B_u} \le \mathrm{e}^{-\ell/8}$. 
So by a union bound, the probability that no event $B_u$ occurs is at least $1-\Delta \mathrm{e}^{-\ell/8} \ge 1-\mathrm{e}^{-\Delta^{\eps/3}}$, assuming that $\Delta$ is large enough.
The conclusion follows.
\end{proof}

Relying on the percolation result stated in Lemma~\ref{lem:percolation}, we can sketch a proof of Theorem~\ref{thm:girth-thawed}. The full-length proof lies in the Appendix.
\begin{proof}[Proof Sketch for Theorem~\ref{thm:girth-thawed}]
We let $g\coloneqq (2+o(1)) \ln \ln n / \ln \Delta$, and assume that the girth is at least $2g+2$ (this holds under the assumption that it is at least $\ln \ln \Delta$ when $\Delta$ is large enough).
We start by fixing a vertex $v$ and sampling from $\bsigma$.
We then deterministically recolour the vertices at distance $g,g-1,g-2,\dots$ from $v$ layer by layer (each layer is an independent set), so as to make the list sizes in the next layer as large as possible. 
If at the end of this process, all $u\in N(v)$ have at least $2$ colours on their list, then it is straightforward to recolour $v$ as we please by first recolouring its neighbours and so $v$ is thawed and clearly we have changed the colours of at most $\Delta^g = \bigO{(\ln n)^2}$ vertices.
Thus $v$ is also $\bigO{(\ln n)^2}$-loose.
Using Corollary~\ref{cor:find_good_col}, we note that after the above process a  vertex $w$ with $\dist(v,w)< g$ will only have a short list if at least $s\coloneqq \Omega(\Delta/\ln \Delta)$ of its children have short lists. 
This allows us to bound from above the probability that $u\in N(v)$ has a short list by the probability that the $s$-upward percolation process on the $(\Delta-1)$-ary tree rooted at $u$ ends with $u$ activated. 
Because the vertices $w$ at distance $g$ from $v$ form an independent set, Theorem~\ref{cor:ind_set_corr} tells us that the distribution of short lists for said $w$ are $\ber(p)$-dominated. 
Thus we can use the percolation bound from Lemma~\ref{lem:percolation} to derive the desired result. 
\end{proof}

\subsection{Constructions}\label{sec:constructions}
One could wonder whether the statement of Theorem~\ref{thm:girth-thawed} (w.h.p. all vertices are simultaneously thawed) can be extended to the case where $G$ is triangle-free (rather than having girth $\ln \ln n$).
We now show that this is not possible, even if the girth of $G$ is an arbitrarily large constant. 

We rely on the following construction from \cite{BBP21}.

\begin{proposition}[Bonamy, Bousquet, Perarnau; 2021]
\label{prop:k-lift}
For every integers $d,g \ge 3$ there exists a $d$-regular graph $G$ of girth at least $g$ that has a proper $(d+1)$-colouring where every vertex is frozen.
\end{proposition}

\begin{proposition}
\label{prop:frozen}
Let $g\ge 3$ be a given integer. Then for every integer $d\ge 3$ and $n$ large enough, there exists an $n$-vertex $d$-regular graph $G$ of girth at least $g$ such that, letting $\bsigma$ be a uniformly random proper $(d+1)$-colouring of $G$, there are w.h.p. $\Theta(n)$ frozen vertices in $\bsigma$ as $n\to \infty$.
\end{proposition}

\begin{proof}
    Let $G_0$ be the $d$-regular graph $G$ of girth $g$ given by Proposition~\ref{prop:k-lift}. We let $n_0\coloneqq |V(G_0)|$, and let $\sigma_0 \in \sC_{d+1}(G_0)$ be such that every vertex $v\in V(G_0)$ is frozen in $\sigma_0$. 
    Let $G$ consist of $n/n_0$ disjoint copies of $G$, and let $\bsigma$ be a uniformly random proper colouring of $G$. For every copy $H$ of $G_0$ within $G$, let $E_H$ be the event that $\restrict{\bsigma}{V(H)}=\sigma_0$. We have $\pr{E_H}\ge 1/(d+1)^{n_0}$, and the random events $(E_H)$ are independent. So the number of copies of $G_0$ in $G$ that are entirely frozen in $\bsigma$ follows the binomial distribution $\mathcal{B}(n/n_0, (d+1)^{-n_0})$. By the standard estimates given by Chernoff-Hoeffding bounds (see Lemma~\ref{lem:chernoff}), the probability that less than $\frac{n}{2n_0(d+1)^{n_0}}$ copies of $G_0$ are entirely frozen is at most $\exp(-\frac{n}{8n_0(d+1)^{n_0}}) \underset{n\to\infty}{\to} 0$.
    We conclude that w.h.p. the number of frozen vertices in $\bsigma$ is at least $\frac{1}{2}(d+1)^{-n_0} \cdot n = \Theta(n)$ as $n\to\infty$.
\end{proof}

\section{Acknowledgement}
A substantial part of this work has been done during the online workshop \emph{Entropy Compression and Related Methods} which took place in March 2021. We are thankful to the organisers, Ross J. Kang and Jean-Sébastien Sereni, and more generally to the Sparse Graph Coalition for making that work possible. 
We are grateful to Felix Joos for proofreading the early versions of that paper.
We also thank Matthieu Rosenfeld and Mike Molloy for insightful discussions.

\bibliographystyle{alpha}
\bibliography{moment}

\appendix

\section{Remaining Proofs and Required Results}

In this section we collect the remaining proofs and results. 
They are written in order of dependency, so that all results have either been cited or proved by the time they are used. 
Let $G$ be a graph and $X\subset V(G)$.
We write $G[X]$ for the subgraph of $G$ induced by $X$. 
We need the following simple observation.

\begin{lemma}
\label{lem:maxdeg}
Let $H$ be a graph, and $L$ a list-assignment of $H$ such that $|L(v)|\ge \deg(v)+1$ for every vertex $v\in V(H)$. 
If $\bsigma$ is a uniformly random proper $L$-colouring of $H$, then given some colour $x\in \bigcup_{v\in V(H)} L(v)$, the probability that $\bsigma(v)\neq x$ for all $v\in V(H)$ is at least
\[ \prod_{v\in V(H)} \pth{1- \frac{1}{|L(v)|-\deg(v)}}.\]
\end{lemma}

\begin{proof}
Since $H$ is greedily $L$-colourable, we can sample a uniform proper $L$-colouring $\bsigma$ of $H$. 
Let $v\in V(H)$.
For every $L$-colouring $\sigma'$ of $H'\coloneqq H\setminus v$, one has
\[ \pr{\bsigma(v)=x \mid \restrict{\bsigma}{H'}=\sigma'} \le \frac{1}{|L(v)|-\deg(v)},\]
since after removing the colours in $\sigma'(N(v))$ from $L(v)$, there remains at least $|L(v)|-\deg(v)$ possible choices for $\bsigma(v)$ which are equiprobable. 
Sampling a uniformly random $L$-colouring and then resampling the colour of each vertex once gives the result, as the resulting random $L$-colouring is also uniformly distributed. 
\end{proof}




The proof of Theorem~\ref{thm:main} will rely on a more advanced version of the Coupon-Collector Lemma that was needed to prove Theorem~\ref{thm:main_simp}. The statement is about the expected number of colours that a uniformly random list-colouring of a given graph $H$ does not use, when the input is a random list-assignment of $H$.

\begin{lemma}[Generalised Coupon-Collector Lemma]
\label{lem:coupon-collector-gen}
Let $H$ be a graph on $n$ vertices, and $\bL$ a random list-assignment of $H$, such that $H$ is deterministically $\bL$-colourable.
Fix some $t\ge 1$, and define the random variable $\bX \coloneqq \#\{v \in V(H) : |\bL(v)| < (\deg(v)+1)(t+1)\}$.
Let $\bsigma$ be a uniformly random $\bL$-colouring of $H$, and define the random variable $L_\bsigma \coloneqq [k] \setminus \bsigma(V(H))$. 
Then
\[
\esp{|L_\bsigma|} \ge k_0\,\mathrm{e}^{-\pth{1+\frac{1}{t}}\frac{n}{k_0}}, \quad \mbox{where } k_0 = k - \esp{\bX}.
\]
\end{lemma}

\begin{proof}
Let us fix a realisation $(L(v))_{v\in V(H)}$ of $\bL$. 
We let $S = \{v : |L(v)| < (\deg(v)+1)(t+1)\}$ be the set of vertices with a short list, and $B = V(H)\setminus S$ be the set of vertices with a big list.
We observe that since $|L(v)| \ge (\deg(v)+1)(t+1)$ for every $v\in B$, we have $\frac{1}{|L(v)|-\deg(v)-1} \le (1+\frac{1}{t})\frac{1}{|L(v)|}$. Combining this observation with one due to Molloy \cite{Mol19}, we obtain
    \begin{equation}
        \label{eq:doublesum2}
        \sum_{x\in [k]} \sum_{\substack{v\in B \\ x\in L(v)}} \tfrac{1}{|L(v)|-\deg(v)-1} \le \pth{1+\tfrac{1}{t}}\sum_{x\in [k]} \sum_{\substack{v\in B \\ x\in L(v)}} \tfrac{1}{|L(v)|} \le \pth{1+\tfrac{1}{t}} \sum_{v\in B} \sum_{x\in L(v)} \tfrac{1}{|L(v)|} \le 
        \pth{1+\tfrac{1}{t}}n.
    \end{equation}

Let us fix the realisation $\sigma_0$ of $\restrict{\bsigma}{S}$, and let $X_0 \coloneqq [k] \setminus \sigma_0(S)$. Note that $|X_0|\ge k-|S|$.
We define $\bX_0 \coloneqq X_0 \setminus \bsigma(B)$; note that $\bX_0$ is precisely $L_\bsigma$ under the condition that $\restrict{\bsigma}{S} = \sigma_0$.  
We let $L_0(v) \coloneqq L(v)\setminus \sigma_0(N(v)\cap S)$ for every $v\in B$; so we have $|L_0(v)| \ge |L(v)| - \deg_S(v)$.
Conditioned on $\restrict{\bsigma}{S}=\sigma_0$, $\bsigma$ induces a uniformly random $L_0$-colouring $\bsigma_1$ of $H[B]$. We apply Lemma~\ref{lem:maxdeg} on $H[B]$ with the list-assignment $L_0$, and obtain

    \begin{align*}
        \esp{|\bX_0|} &= \sum_{x\in X_0} \pr{x\notin \bsigma_1(B)}\\
        & \ge \sum_{x\in X_0} \prod_{\substack{v\in B \\ x\in L_0(v)}} \pth{1-\frac{1}{|L_0(v)|-\deg_B(v)}} \ge \sum_{x\in X_0} \prod_{\substack{v\in B \\ x\in L(v)}} \pth{1-\frac{1}{|L(v)|-\deg(v)}} \hspace{-1000pt}
         \\
        &\ge \sum_{x\in X_0} \exp \pth{-\sum_{\substack{v\in B \\ x\in L(v)}} \frac{1}{|L(v)|-\deg(v)-1}} &
        \mbox{since $1-\frac{1}{z} > \mathrm{e}^{-\frac{1}{z-1}}$ for every $z>1$;}\\
        &\ge |X_0| \exp\pth{-\frac{1}{|X_0|}\sum_{x\in X_0}\sum_{\substack{v\in B \\ x\in L(v)}} \frac{1}{|L(v)|-\deg(v)-1}} &
        \mbox{by convexity of $\exp$;}\\
        &\ge (k-|S|) \exp\pth{- \frac{\pth{1+\frac{1}{t}}n}{k-|S|}} &
        \mbox{by \eqref{eq:doublesum2}.}
    \end{align*}

    There remains to average over all possible realisations of $\bL$. We use Jensen's inequality together with the convexity of the function $z\mapsto z\mathrm{e}^{-C/z}$ for every $C>0$ over the interval $(0,+\infty)$, and obtain that 
    \begin{align*}
        \esp{|\bL|} \ge \esp{(k-\bX) \exp\pth{- \frac{\pth{1+\frac{1}{t}}n}{k-\bX}}} \ge \esp{(k-\bX)}\exp\pth{- \frac{\pth{1+\frac{1}{t}}n}{\esp{k-\bX}}} = k_0 \mathrm{e}^{-\pth{1+\frac{1}{t}}\frac{n}{k_0}},
    \end{align*}
    where $k_0 = k - \esp{\bX}$.

\end{proof}

Before proceeding with the formal proof of Theorem~\ref{thm:main}, we describe the random experiment at its core.
Let $G$ be a graph and $v\in V(G)$. We wish to prove that for $\bsigma$ drawn uniformly at random from the set of proper $k$-colourings of $G'\coloneqq G\setminus v$, the list of available colours at $v$ given $\bsigma$ is large in expectation. 
We do so by analysing the following random procedure.
\begin{enumerate}[(i)]
    \item \label{one} Sample a proper $k$-colouring of $G'$ uniformly at random;
    \item \label{two} mark all vertices in $N(v)$ that have short lists \emph{due to the colouring on} $G'\setminus N(v)$;
    \item \label{three} uncolour all unmarked vertices in $N(v)$;
    \item \label{four} choose a proper re-colouring of the uncoloured vertices in $N(v)$ uniformly at random.
\end{enumerate}
Conveniently, the random proper colouring obtained at the end of this experiment is once again uniformly distributed across all proper colourings of $G'$. 
This is because the vertices we mark in step \eqref{two} are selected solely based on the colouring of $G'\setminus N(v)$, which remains fixed after step \eqref{one}.
The proof proceeds by computing a lower bound on the expected size of the list of available colours at $v$ after step \eqref{four}, thus proving the desired induction hypothesis.
We restate the theorem for the reader's convenience. 
\subsection*{Theorem~\ref{thm:main}}
\textit{Let $G$ be an $n$-vertex graph of maximum degree $\Delta$ such that every graph induced by a neighbourhood in $G$ has average degree at most $d\le \frac{\Delta}{6}-1$. Write $\rho \coloneqq \Delta/(d+1)$ and let $\ell \ge (d+1)(\ln \rho)^3$. Then for every list-assignment  $L\colon V(G)\to 2^\mathbb{N}$ with 
\[|L(v)| \ge \pth{1+\frac{2}{\ln \rho}} \frac{\deg(v)}{W\pth{\frac{\deg(v)}{\ell}}}\]
for every vertex $v\in V(G)$, the following holds.
For all $v \in V(G)$, the uniformly random proper $L$-colouring $\bsigma$ of $G\setminus v$ satisfies $\esp{\ell_\bsigma(v)} \ge  \ell$.}

\begin{proof}[Proof of Theorem~\ref{thm:main}]
Fix $\rho\coloneqq \frac{\Delta}{d+1} \ge 6$, $t\coloneqq (d+1)(\ln \rho+1)$, and $\ell \ge (d+1)(\ln \rho)^3$.
 For every $v\in V(G)$, let $k(v) \coloneqq \pth{1+\frac{2}{\ln \rho}}\frac{\deg(v)}{W\pth{\frac{\deg(v)}{\ell}}}$. Note that $k(v)\ge \frac{\deg(v)}{W\pth{\frac{\deg(v)}{\ell}}} = \ell \mathrm{e}^{W(\deg(v)/\ell)} \ge \ell$, because $W(x)\ge 0$ for every $x\ge 0$.

\smallskip

Let $L$ be any list-assignment of $G$ such that $|L(v)| \ge k(v)$ for every vertex $v\in V(G)$.
Let $\sC(H)$ denote the set of proper $L$-colourings of $H$.
We show by induction that for all induced subgraphs $H\subseteq G$ we have 
\begin{equation}
\label{eq:HI}
\tag{$\star$}
    |\sC(H)|\ge \ell \, |\sC(H\setminus v)|,
\end{equation}
for all $v\in V(H)$. Observe that, given a colouring $\sigma\in \sC(H\setminus v)$, the number of extensions of $\sigma$ to a colouring in $\sC(H)$ is precisely $\ell_\sigma(v)$. Hence \eqref{eq:HI} is equivalent to $\esp{\ell_{\bsigma}(v)}\ge \ell$, for a uniformly random colouring $\bsigma \in \sC(H\setminus v)$. 

The base case with $H=v$ for some $v\in V(G)$ follows as we have $|\sC(\varnothing)|=1$, and $|\sC(H)|=k(v)\ge \ell$.
Suppose now that $|V(H)|\ge 2$, and that the induction hypothesis \eqref{eq:HI} holds for all induced subgraphs of $H'\coloneqq H\setminus v$.
We let $\bsigma$ be drawn uniformly at random from $\sC(H')$, and we let $\bsigma_0$ be obtained from $\bsigma$ by uncolouring $N_H(v)$; thus $\bsigma_0$ is a proper $L$-colouring of $H_0\coloneqq H'\setminus N(v)$.
For every $u\in N_H(v)$, we denote by $d_u$ the degree of $u$ within $H[N_H(v)]$, and we let $t_u \coloneqq (d_u+1)(\ln \rho + 1)$. Hence the average of $t_u$ over all $u\in N_H(v)$ is at most $t$.
Given the realisation of $\bsigma$, we say that a neighbour $u\in N(v)$ of $v$ has a short list if $\ell_{\bsigma_0}(u) \le t_u$, and we let $S_\bsigma$ be the set of vertices $u\in N(v)$ with short lists. First we observe that the expected size of $S_\bsigma$ is small.
By the induction hypothesis \eqref{eq:HI}, and using again the observation that the number of extensions of a colouring $\sigma$ to an additional vertex $u$ is $\ell_\sigma(u)$, we know that 
\begin{equation}
\label{eq:markov}
      \pr{\ell_\bsigma(u)\le t_u} = \frac{\#\{\sigma'\in \sC(H') : \ell_{\sigma'}(u)\le t_u\}}{|\sC(H')|}\le \frac{t_u \cdot |\sC(H'\setminus u)|}{\ell \cdot |\sC(H'\setminus u)|} \le \frac{t_u}{\ell}.
  \end{equation} 
Summing \eqref{eq:markov} over all $u\in N(v)$, and since $\ell_\bsigma(u) \le \ell_{\bsigma_0}(u)$ for every $u\in N(v)$, we obtain that 
\begin{align}
\esp{|S_\bsigma|} &= \sum_{u\in N(v)}\limits  \pr{\ell_{\bsigma_0}(u)\le t_u} \le \sum_{u\in N(v)}\limits  \pr{\ell_{\bsigma}(u)\le t_u}
\le \sum_{u\in N(v)}\limits \frac{t_u}{\ell}
\\ &\le \frac{t \deg(v)}{\ell}  \le \frac{1}{\ln \rho}\;\frac{\deg(v)}{\ln \rho - 1 }
\le \frac{1}{\ln \rho}\; \frac{\deg(v)}{W\pth{\frac{\deg(v)}{\ell}}},
\label{eq:esperance(k)}
\end{align}
where we use that $\ln \rho -1 \ge W(\rho/(\ln \rho)^3) \ge W(\deg(v)/\ell)$, since $\rho \ge 6$. To see this, observe that the function $x\mapsto \ln x - 1 - W(x/(\ln x)^3)$ is increasing when $x\in (1,+\infty)$, and has a positive value at $x=6$.

We are now going to apply Lemma~\ref{lem:coupon-collector-gen} to the graph $H[N(v)]$ on $\deg(v)$ vertices, with the random list-assignment $L_{\bsigma_0}$, where the colours have been renamed in such a way that $L(v)=[k(v)]$.
We let $\bsigma_1$ be a uniformly random $L_{\bsigma_0}$-colouring of $H[N(v)]$; by Lemma~\ref{lem:coupon-collector-gen} we have 
\begin{equation}
    \label{eq:coupon-collector}
    \esp{\ell_{\bsigma_1}(v)} \ge k_0 \mathrm{e}^{-\pth{1+\frac{1}{\ln \rho}}\frac{\deg(v)}{k_0}}, 
\end{equation}
where $k_0 = k(v) - \esp{|S_{\bsigma}|} \ge \pth{1+\frac{1}{\ln \rho}} \frac{\deg(v)}{W\pth{\frac{\deg(v)}{\ell}}}$. We observe that, by construction,  $\bsigma_0 \cup \bsigma_1$ and $\bsigma$ are identically distributed, hence we have 
\begin{align*}
\esp{\ell_{\bsigma}(v)} = \esp{\ell_{\bsigma_1}(v)}
\ge \pth{1+\frac{1}{\ln \rho}}\frac{\deg(v)}{W\pth{\frac{\deg(v)}{\ell}}} \mathrm{e}^{-W\pth{\frac{\deg(v)}{\ell}}} =\pth{1+\frac{1}{\rho}} \ell.
\end{align*}
This ends the proof of the induction.
\end{proof}

We now prove that Theorem~\ref{thm:main} directly implies Corollary~\ref{cor:vu}.

\subsection*{Corollary~\ref{cor:vu}}
\textit{Let $G$ be a graph of maximum degree $\Delta$, such that every neighbourhood spans at most $\Delta^2/f$ edges, for some $1 \le f \le \Delta^2+1$. Then
$ \chi_\ell(G) \le (1+o(1)) \frac{\Delta}{\ln \min \{\Delta,f\}}$ as $f\to \infty$.}
\begin{proof}[Proof of Corollary~\ref{cor:vu}]
Let $G$ satisfy the hypothesis of Corollary~\ref{cor:vu}.
It is well-known that there exists a $\Delta$-regular graph $H$ and a mapping $\varphi\colon V(H) \to V(G)$ such that $G$ is an induced subgraph of $H$, and for every vertex $v\in V(H)$ the number of edges in $H[N(v)]$ equals that in $G[N(\varphi(v))]$ (see for instance the construction in \cite[Lemma 6]{PiSe21}). So we may assume that $G$ is regular.
The average degree in $G[N(v)]$ is therefore at most $d \coloneqq 2\Delta/f$, for every $v\in V(G)$. We have 
\vspace{-4pt}
 \[ \frac{\Delta}{d+1} = \frac{\Delta f}{f + 2\Delta} \ge \begin{cases}  \frac{f}{3} & \mbox{if $f \le \Delta$,} \\ \frac{\Delta}{3} & \mbox{if $f \ge \Delta$.} \end{cases}
\vspace{-2pt} \]
 
 Setting $\rho\coloneqq \frac{\Delta}{d+1} \ge \min \left\{f/3,\Delta/3\right\}$, we let $L\colon V(G)\to 2^\mathbb{N}$ be any list-assignment of $G$ with $|L(v)| \ge (1+2/\ln \rho) \frac{\Delta}{W\pth{\rho/(\ln \rho)^3}}$. By Theorem~\ref{thm:main}, $G$ is $L$-colourable. Hence
 \[ \chi_\ell(G) \le \pth{1+\frac{2}{\ln \rho}} \frac{\Delta}{W\pth{\frac{\rho}{(\ln \rho)^3}}} 
\le (1+o(1)) \frac{\Delta}{\ln \min \{\Delta,f\}},\]
 as $f \to \infty$ (and therefore also $\Delta\to \infty$).
\end{proof}

Next we prove Theorem~\ref{thm:count}, which yields a lower bound on the number of colourings of a graph in the setting of Theorem~\ref{thm:main}.

\subsection*{Theorem~\ref{thm:count}}
\textit{Let $G$ be an $n$-vertex graph of maximum degree $\Delta$ such that every graph induced by a neighbourhood in $G$ has average degree at most $d \le \frac{\Delta}{6}-1$. Let $f \coloneqq \Delta/(d+1)$ and
suppose $L\colon V(G)\to 2^\mathbb{N}$ is a list-assignment  with $|L(v)| \ge \pth{1+\frac{1}{\ln \rho}} q(v)$, where  
\[q(v) \ge \pth{1+\frac{1}{\ln \rho}} \frac{\deg(v)}{W\pth{\frac{\deg(v)}{(d+1)(\ln \rho)^3}}}\]
 for every vertex $v\in V(G)$.
Then there are at least $(q\big/\sqrt{D/(d+1)})^n$
 proper $L$-colourings of $G$, where $D$ is the geometric mean of the degrees in $G$, and $q$ is the geometric mean of $\{q(v)\}_{v\in V(G)}$.}

\begin{proof}[Proof of Theorem~\ref{thm:count}]
Following the same set-up as in the proof of Theorem~\ref{thm:main} we have 
$k_0 = k(v) - \esp{|S_\bsigma|} \ge q(v)$, for every vertex $v\in V(H)$.
We also note that in \eqref{eq:coupon-collector}, we can replace $\deg(v)$ with $\deg_H(v)$. 
Then
\begin{align}
\label{eq:tight-expectancy}
\esp{\ell_\bsigma(v)} &\ge k_0\mathrm{e}^{-\pth{1+\frac{1}{\ln \rho}}\frac{\deg_H(v)}{k_0}} \ge q(v)\mathrm{e}^{-\pth{1+\frac{1}{\ln \rho}}\frac{\deg_H(v)}{q(v)}}.
\end{align}

We now let $v_1,\ldots,v_n$ be an ordering of $V(G)$ such that $\pth{q(v_i)}_{i=1}^n$ is non-decreasing. Letting $H_1$ be the empty graph, and $H_i \coloneqq G[v_1, \ldots, v_{i-1}]$ for every $2\le i \le n$, we apply \eqref{eq:tight-expectancy} on the pairs $(H_i,v_i)$ for every $1\le i \le n$ and obtain that the number of $L$-colourings of $G$ is
\begin{align*}
|\sC(G)| &\ge \prod_{i=1}^n q(v_i) \mathrm{e}^{-\pth{1+\frac{1}{\ln \rho}}\deg_{H_i}(v_i)/q(v_i)} \\
& = q^n \exp \pth{-\pth{1+\tfrac{1}{\ln \rho}} \sum_{uv \in E(G)}\limits \min \left\{ \frac{1}{q(u)}, \frac{1}{q(v)} \right\}} \\
&\ge q^n \exp \pth{-\pth{1+\tfrac{1}{\ln \rho}} \sum_{uv \in E(G)}\limits \pth{ \frac{1}{2q(u)} + \frac{1}{2q(v)} }} \\
&= q^n  \exp\pth{-\pth{1+\tfrac{1}{\ln \rho}} \sum_{i=1}^n \limits\frac{\deg(v_i)}{2q(v_i)}} \\
& \ge  q^n  \exp\pth{-\frac{1}{2} \sum_{i=1}^n \limits \ln \frac{\deg(v_i)}{d+1}} = q^n \pth{\frac{D}{d+1}}^{-n/2}.
\qedhere
\end{align*}
\end{proof}

Let us recall the statement of Lemma~\ref{lem:MOL-list-conc}.
\subsection*{Lemma~\ref{lem:MOL-list-conc}}
\textit{
Let $G$ be a triangle-free graph, let $v \in V(G)$ and let $\sigma_0$ be a proper $k$-colouring of $G\setminus N[v]$ (for some $k$), with at least one extension to $G$. 
Then if $\bsigma$ is the uniformly random extension of $\sigma_0$ to $G$, writing $\ell \coloneqq \esp{\ell_{\bsigma}(v)}$, we have 
\[
\pr{\ell_{\bsigma}(v) \le (1-\delta)\ell} \le \mathrm{e}^{-\frac{\delta^2\ell}{2}},
\]
for all $\delta \in (0,1)$.
}

\medskip
This is proved in the context of proper \emph{partial} $k$-colourings in \cite[Lemma~7]{Mol19}. For completeness, we repeat the proof which holds almost readily in the context of Lemma~\ref{lem:MOL-list-conc}.
In particular, we will need the Chernoff-Hoeffding bound stated in \cite[Lemma~3(b)]{Mol19}. We say that a set of binary random variables $(\bY_i)_{i\in [s]}$ are negatively correlated if 
\[ \esp{\prod_{i\in [s]} \bY_i} \le \prod_{i\in [s]} \esp{\bY_i}.\]

\begin{lemma}
\label{lem:chernoff}
Suppose $\bX_1,\dots,\bX_s$ are binary random variables. Set $\bY_i \coloneqq 1-\bX_i$, and $\bX\coloneqq \sum_{i\in [s]} \bX_i$. Write $ \mu \coloneqq\esp{\bX}$. If $(\bY_i)_{i\in [S]}$ are negatively correlated, then for all $0<\delta<1$,
\[
\pr{\bX \le (1-\delta)\mu} \le e^{-\frac{\delta^2 \mu}{2}}.
\]
\end{lemma}

\begin{proof}[Proof of Lemma~\ref{lem:MOL-list-conc}]
    For every $x\in L(v)$, we denote $E_x$ the event that $x\notin L_\bsigma(v)$. We first argue that the events $(E_x)_{x\in L(v)}$ are negatively correlated, i.e. 
    \begin{equation}
        \label{eq:negative-correlation}
        \pr{\bigwedge_{x\in X} E_x} \le \prod_{x\in X} \pr{E_x},
    \end{equation}
    for every $X\subseteq L(v)$. First observe that if $\pr{E_x}=0$ for some $x\in X$, then \eqref{eq:negative-correlation} trivially holds since both terms equal zero. Second, if $\pr{E_x}=1$ for some $x\in X$, then \eqref{eq:negative-correlation} is equivalent to its statement when we remove $x$ from $X$.
    So, in order to prove \eqref{eq:negative-correlation}, let us show that $\pr{E_x ~\middle| ~ \bigwedge_{y\in Y} E_y} \le \pr{E_x}$, for every $x\in L(v)$ and $Y\subseteq L(v)\setminus x$, where we assume that $0<\pr{E_x}<1$. This is equivalent to 
    \begin{equation}
        \label{eq:coupling}
        \pr{\bigwedge_{y\in Y} E_y ~\middle| ~\overline{E_x}} \ge \pr{\bigwedge_{y\in Y} E_y}.
    \end{equation}
    We show that \eqref{eq:coupling} holds with a coupling argument\footnote{This coupling did not appear in earlier proofs of this statement, \eqref{eq:coupling} was assumed.}.
    For every $u\in N(v)$, we define $\bsigma'(u)\coloneqq \bsigma(u)$ if $\bsigma(u) \neq x$; otherwise $\bsigma'(u)$ is drawn uniformly at random from $L(u)\setminus x$ (this is non-empty, otherwise we would have $\pr{E_x}=1$). Then $\bsigma'$ follows the distribution of $\bsigma$ under the condition $\overline{E_x}$. Moreover, by construction, whenever we have $Y\subseteq \bsigma(N(v))$, we deterministically have $Y\subseteq \bsigma'(N(v))$. So 
    \begin{align*}
        \pr{\bigwedge_{y\in Y} E_y ~\middle| ~\overline{E_x}} = \pr{Y \subseteq \bsigma'(N(v))}\ge \pr{Y \subseteq \bsigma(N(v))} = \pr{\bigwedge_{y\in Y} E_y},
    \end{align*}
    as desired.

    To finish the proof, we apply Lemma~\ref{lem:chernoff}.
    We have that $\ell_\bsigma(v)$ is the number of colours $x\in L(v)$ such that $E_x$ does not hold, and since we have shown that the events $(E_x)_{x\in L(v)}$ are negatively correlated, we infer that
    \[\pr{\ell_\bsigma(v) \le (1-\delta)\ell} \le \mathrm{e}^{-\frac{\delta^2\ell}{2}},\]
    as desired.
\end{proof}

We finish with a fully-detailed proof of Theorem~\ref{thm:girth-thawed}, which is our central result on the geometry of the solution space for colourings of high girth graphs.  

\subsection*{Theorem~\ref{thm:girth-thawed}}
\textit{ For all $\eps>0$ the following holds for all $k$ sufficiently large and $\Delta \le (1-\eps)k\ln k$. 
 Suppose $G$ is a graph on $n$ vertices with maximum degree $\Delta$ and girth at least $\ln\ln n$.
 Then for the uniformly random proper $k$-colouring $\bsigma$ of $G$, the following holds w.h.p. for every $v \in V(G)$:
 \begin{enumerate}
    \item in $\bsigma$, $v$ is $\bigO{(\ln n)^2}$-loose\footnote{With a more careful (and technical) analysis, the authors believe that this could be replaced with $O(\ln n)$-loose.};
    \item in $\bsigma$, $v$ is thawed.
\end{enumerate}}

\begin{proof}[Proof of Theorem~\ref{thm:girth-thawed}]
Let us fix a vertex $v \in V(G)$.
We write $s \coloneqq \frac{\eps k}{5} \ge \frac{\eps}{5}\frac{\Delta}{\ln \Delta}$ and $g \coloneqq \lceil \frac{2 \ln \ln n^3}{\ln (s-1)} + 1\rceil$. When $\Delta$ is large enough, we have $2g+2\le \ln \ln n$.
Let $U_i$ denote the set of vertices at distance exactly $i$ from $v$. 
Thus $U_1=N(v)$, and since the girth of $G$ is at least $2g+2$, it holds that $U_i$ is an independent set for every $i \le g$.
Let $\sigma_0$ be chosen uniformly at random from $\sC_k(G)$.
We will show that, with probability at least $1 - 1/n^2$, we can find a sequence $\sigma_1,\dots,\sigma_{g-1}\in \sC_k(G)$ with the following properties: 
\begin{enumerate}
    \item for all $i \in [g-1]$, $\restrict{\sigma_i}{V(G)\setminus U_{g-i+1}} = \restrict{\sigma_{i-1}}{V(G)\setminus U_{g-i+1}}$;
    \item $\ell_{\sigma_{g-1}}(u) > 5/\eps$ for all $u \in N(v)$.
\end{enumerate}
We construct our sequence (deterministically) as follows. 
Given $\sigma_{i-1}$, for $i \in [g-1]$, we let $\sigma_i \in C_k(G)$ be such that $\ell_{\sigma_i}(w)$ is maximum for all $w \in U_{g-i}$, given that $\restrict{\sigma_i}{V(G)\setminus U_{g-i+1}} = \restrict{\sigma_{i-1}}{V(G)\setminus U_{g-i+1}}$.
We note that we can simultaneously maximise $\ell_{\sigma_i}(w)$ for all $w \in U_{g-i}$ because every vertex in $U_{g-i+1}$ is adjacent to exactly one vertex of $U_{g-i}$.

We first deal with the case where $G$ is $\Delta$-regular and for convenience we write $\ell_i(u)$ for $\ell_{\sigma_i}(u)$. 
Now suppose $\ell_i(u) \le 5/\eps$ for some $u \in U_{g-i}$ (we say that $u$ has a short list in $\sigma_i$).
By our definition of $\sigma_i$, there must be no colouring $\tau\in \sC_k(G)$, such that $\tau(w) = \sigma_{i-1}(w)$ for all $w \in V(G)\setminus U_{g-i+1} $ and $\ell_\tau(u) > 5/\eps$. 
When $\Delta$ is large enough, we have $\Delta^{\eps/2}>5/\eps$, so by Corollary~\ref{cor:find_good_col} we must have $\ell_{i-1}(w)\le 5/\eps$ for at least $s$ neighbours $w$ of $u$, $s-1$ of which must lie in $U_{g-i+1}$.

We now observe the following.
Fix some $r \in N(v)$ and let $T$ be the $(\Delta -1)$-ary tree of depth $g-1$ obtained by deleting $v$ and then looking at the induced subgraph of $G-v$ within distance $g-1$ of $r$.
Let $X\coloneqq \{ w \in U_g \cap V(T): \ell_0(w) \le 5/\eps\}$ be the leaves of $T$ that have short lists in $\sigma_0$. 
Then $\ell_i(u) \le 5/\eps$ only if $u$ is eventually activated in the $(s-1)$-upward percolation on $T$ where the set of leaves that are initially activated is $X$.

We know that $U_g$ is an independent set, so by Theorem~\ref{cor:ind_set_corr} the events that $\ell_0(w) < 5/\eps$ for $w \in U_g$ are $\ber(p)$-dominated for $p = \frac{5}{\eps \ell}$, where $\ell = \Delta^{\eps/2}$.
Hence, when $\Delta$ is large enough, we have $s-1 \ge \max \{ 6p\Delta, 3\ln \Delta\}$.
We apply Lemma~\ref{lem:percolation} with $f = g-1$ and obtain that 
\begin{align*}
    \pr{\ell_{g-1}(r) \le 5/\eps} &\le \pr{r \text{ is activated}} \le \exp\pth{-{(s-1)^{f/2}}} \\
    &\le  \exp\pth{-{\mathrm{e}^{\frac{1}{2}(g-1)\ln (s-1)}}} \le \exp\pth{-{\mathrm{e}^{\ln \ln n^3}}} = \frac{1}{n^3}.
\end{align*}
Union bounding over all neighbours $r$ of $v$ we find that 
\[
\pr{\ell_{g-1}(r) > 5/\eps \mbox{ for every } r \in N(v) } \ge 1 - \frac{\Delta}{n^3} > 1 - \frac{1}{n^2}.
\]
 We do a union bound over all $v\in V(G)$, and obtain that with probability at least $1-1/n$ we find the desired sequence $\sigma_0,\dots,\sigma_{g-1}$ for every $v\in V(G)$.

We now observe that given such a colouring $\sigma_{g-1}$, we can change the colour of $v$ to any other by first recolouring each $r \in N(v)$ to avoid the desired colour. 
Thus the existence of $\sigma_0,\dots,\sigma_{g-1}$ implies that $v$ is thawed in $\sigma_0$ and is $\bigO{(\ln n)^2}$-loose in $\sigma_0$, since we only need to recolour vertices at distance at most $g$ from $v$, and there are at most $\Delta^g = \bigO{(\ln n)^2}$ of them.
Union bounding over all choices of $v$ yields the result.

Finally we deal with the case where $G$ is not regular. 
For each vertex $w$ at distance less than $g+1$ from $v$, with degree less than $\Delta$ we simply add rooted $(\Delta-1)$-ary trees of depth $g$ and connect $w$ to the root, until $w$ has degree $\Delta$. 
We then treat all newly added edges as invisible in that they can be monochromatic. 
Thus the lists of the newly added vertices are always just $[k]$, and their colours do not affect the lists of any other vertices. 
Therefore we once again have the crucial property that if  some vertex $u\in U_{g-i}$ has  $\ell_i(u) \le 5/\eps$, then at least $s$ of its children $w$ must have $\ell_{i-1}(w)\le 5/\eps$. 
Further the events that the leaves of $T$ have short lists are once again $\ber(p)$ dominated for the same value of $p$.
This allows us to repeat the rest of the arguments. 

\end{proof}

\end{document}